\documentclass[10pt]{amsart}
\usepackage{latexsym}
\usepackage{amscd, amsfonts, eucal, mathrsfs, amsmath, amssymb, amsthm}
\usepackage[OT2,OT1]{fontenc}
\usepackage{hyperref}
\input xy
\xyoption{all}

\usepackage{fullpage}
\usepackage{newcent}

\newcommand{\field}[1]{\mathbf #1}

\newcommand{\mc}[1]{\mathcal #1}
\newcommand{\ms}[1]{\mathscr #1}
\newcommand{\widebar}[1]{\overline{#1}}
\newcommand{\Ab}{\mathbf{Ab}}

\def\risom{\overset{\sim}{\rightarrow}}

\def\limr{\varinjlim}

\newcommand{\R}{\field R}

\newcommand{\Z}{\field Z}

\newcommand{\PP}{\mathbb P}
\newcommand{\ZZ}{\mathbb Z}

\newcommand{\simto}{\stackrel{\sim}{\to}}

\newcommand{\shom}{\ms H\!om}

\newcommand{\saut}{\ms A\!ut}

\DeclareMathOperator{\clean}{cl}
\newcommand{\Spec}{\operatorname{Spec}}
\newcommand{\spec}{\operatorname{Spec}}

\newcommand{\rspec}{\operatorname{\bf Spec}}

\renewcommand{\P}{\field P}

\newcommand{\A}{\field A}

\DeclareMathOperator{\Pic}{Pic}

\DeclareMathOperator{\pr}{pr}

\newcommand{\m}{\boldsymbol{\mu}}

\newcommand{\G}{\field G} 

\newcommand{\et}{\operatorname{\acute{e}t}}

\renewcommand{\H}{\operatorname{H}}
\newcommand{\vH}{\check\H}
\newcommand{\bands}{L}

\DeclareMathOperator*{\tensor}{\otimes}

\newcommand{\surj}{\twoheadrightarrow}

\newcommand{\inj}{\hookrightarrow}

\newcommand{\id}{\operatorname{id}}

\DeclareMathOperator{\Aut}{\operatorname{Aut}}
\DeclareMathOperator{\aut}{Aut}

\DeclareMathOperator{\Isom}{\operatorname{Isom}}
\DeclareMathOperator{\isom}{Isom}
\DeclareMathOperator{\Hom}{\operatorname{Hom}}
\DeclareMathOperator{\Bin}{Bin}
\DeclareMathOperator{\Tri}{Trin}
\DeclareMathOperator{\Conj}{Conj}

\DeclareMathOperator{\Br}{\operatorname{Br}}

\newcommand{\obj}{\operatorname{Obj}}
\newcommand{\Obj}{\obj}

\newcommand{\Stacks}{\mathbf{Stacks}}
\newcommand{\Func}{\mathbf{Func}}

\renewcommand{\split}{\operatorname{split}}

\DeclareMathOperator{\B}{\operatorname{\sf B\!}}

\newcommand{\fingps}{\mathbf{FinGps}}
\newcommand{\Sch}{\textrm{-}\mathbf{Sch}}
\newcommand{\Set}{\mathbf{Set}}
\DeclareMathOperator{\sym}{Sym}

\newtheorem{lem}{Lemma}[subsection]
\newtheorem{thm}[lem]{Theorem}

\newtheorem{prop}[lem]{Proposition}

\newtheorem{cor}[lem]{Corollary}

\theoremstyle{definition}
\newtheorem{defn}[lem]{Definition}
\newtheorem*{definition}{Definition}

\newtheorem*{convention}{Convention}
\newtheorem{example}[lem]{Example}
\newtheorem*{Examples}{Examples}

\theoremstyle{remark}
\newtheorem{remark}[lem]{Remark}

\newtheorem{notn}[lem]{Notation}

\newtheorem*{question}{Question}

\begin{document}

\title{Functorial reconstruction theorems for stacks}
\author{Max Lieblich and Brian Osserman}

\begin{abstract}
We study the circumstances under which one can reconstruct a stack
from its associated functor of isomorphism classes.  This
is possible surprisingly often: we show that many of the standard
examples of moduli stacks are determined by their functors. Our
methods seem to exhibit new anabelian-type phenomena, in the form of
structures in the category of schemes that encode automorphism data
in groupoids.
\end{abstract}

\maketitle
\tableofcontents
\section{Introduction}

Let $S$ be a scheme.  Write $\Stacks_S$ for the $1$-category
underlying the $2$-category of fppf $S$-stacks with small fiber
categories and $\Func_S$ for the category
of set-valued contravariant functors on the category of $S$-schemes.
 
There is a natural functor
$$F:\Stacks_S\to\Func_S$$
which sends a stack $\ms S$ to its associated functor $F_{\ms S}$ of
isomorphism classes (so that $F_{\ms S}(T)$ is the set of
isomorphism classes of objects of $\ms S_T$).

Given a category $\mc C$, call a subclass $P$ of
$\operatorname{Obj}\mc C$ a {\it property\/} if it is closed under
isomorphism.  Given a property $P$ of $\mc C$ and a functor $F:\mc
C\to\mc D$, there is a {\it pushforward property\/} $F_{\ast}P$
consisting of all objects of $\mc D$ isomorphic to an object of
$F(P)$.  There is an associated category $\hom(\mc C)$ consisting of
diagrams in $\mc C$ of the form $c\to d$; a functor $F:\mc C\to\mc D$
induces a functor $\hom(\mc C)\to\hom(\mc D)$, which we will also
denote $F$ (by abuse of notation).  A property of $\hom(\mc C)$ will
also be called a {\it property of morphisms in $\mc C$\/}.

\begin{definition}
  Given a subcategory $\mc C$ of $\Stacks_S$, a property $P$ of
  $\mc C$ (resp.\ $\hom(\mc C)$) is {\it $\mc C$-isonatural\/} if
  $P=(F|_{\mc C})^{-1}((F|_{\mc C})_{\ast}P)$.
\end{definition}

It is straightforward that a property $P$ is isonatural if and only if
there is some property $Q$ of $\Func_S$ (resp.\ $\hom(\Func_S)$) such
that $P=F^{-1}(Q)$.  The question which we address in this paper is
the following.

\begin{question}\label{sec:introduction}
  Which properties of $\Stacks_S$ (resp.\ $\hom(\Stacks_S)$) are
  isonatural (resp.\ $\hom(F)$)?
\end{question}

\begin{Examples}
(1)  Given a stack $\ms X$, there is a property $[\ms X]$ consisting of
  the stacks isomorphic to $\ms X$.  The statement that $[\ms X]$ is
  isonatural is the same as the statement that a stack $\ms Y$ is isomorphic
  to $\ms X$ if and only if $F_{\ms Y}$ is isomorphic to $F_{\ms X}$.
  In other words, $\ms X$ is characterized up to $1$-isomorphism by
  its associated functor of isomorphism classes.  In this case, we
  will say ``$\ms X$ is isonatural'' in place of ``$[\ms X]$ is isonatural.''

  (2) The subcategory of Deligne-Mumford stacks yields a property of
  $\Stacks_S$.  The statement that this property is isonatural is the
  same as the statement that if $\ms X$ is Deligne-Mumford and $F_{\ms
    X}$ is isomorphic to $F_{\ms Y}$ then $\ms Y$ is Deligne-Mumford.
  In other words, there is a functorial criterion for a stack to be
  Deligne-Mumford.

  (3) The subcategory of $\hom(\Stacks_S)$ parametrizing
  representable, smooth, etc., morphisms $\ms X\to\ms Y$ defines a
  property.  Isonaturality means that this a morphism $\ms X\to\ms Y$
  has this property if and only if the induced map $F_{\ms X}\to
  F_{\ms Y}$ has some other property (in the formal sense), which one
  would ideally like to describe.
\end{Examples}

For all of the properties and stacks that we can prove are isonatural,
we explicitly describe the corresponding properties of the associated
functors (resp.\ the morphisms of associated functors).  This
constructive aspect of our proofs requires that we restrict our
attention to a particular subcategory $\mc Q$ of $\Stacks_S$, which we call
\emph{quasi-algebraic stacks\/}.  The reader is
referred to Definition \ref{def:quasi-algebraic} for a glimpse of this
subcategory.  

\begin{convention}
  In the rest of this paper, ``isonatural'' will mean ``$\mc
  Q$-isonatural.''
\end{convention}

At first glance, it may seem that there is no hope of recovering the
automorphism data contained in the stack after passing to $F_{\ms X}$,
but it turns out that this is not the case.  Theorem
\ref{thm:main-examples} asserts that many classical moduli stacks are in
fact isonatural.  More generally, Theorem \ref{thm:summary} offers
substantial evidence for a positive answer to the following question.

\begin{question}
  Are all quasi-algebraic stacks isonatural?
\end{question}

As we show in Section \ref{sec:context}, this is not a purely abstract
statement about stacks on sites (even if one considers only stacks
with representable diagonals), as there are many examples of stacks
on the small \'etale site of a field or geometrically unibranch scheme
which are not isonatural.

To the reader used to thinking about the theory of moduli (and who has
learned the standard phylogeny which proceeds from functors to sheaves to 
stacks), the 
results we present here may seem surprising. If the information contained 
in the category fibered in groupoids is not contained in the isomorphism
data (as we are taught), then where is it? The answer, of course, lies
in descent theory, which creates a tight 
relationship between the sets of isomorphism classes of objects and their 
automorphism groups.

The situation is evocative of anabelian geometry. In anabelian theory, 
and its subsequent extensions by Mochizuki, auxilliary categories -- the 
category of finite \'etale covers \cite{mo6} \cite{ta3}, or the slice 
category of (log) schemes \cite{mo7} -- can be shown to determine a scheme,
typically by explicit reconstructive arguments.
Our results show, roughly, that structures in the category of $S$-schemes 
can serve to reconstruct groupoids from the associated ``coarse'' data 
contained in functors.

A basic example of this kind of structure arises 
in studying the classifying stack $\B{G}$ over an algebraically
closed field $k$, where $G$ is a finite group.  As we show in Proposition
\ref{sec:classifying-stacks}, there is a faithful functor $\beta$ from the
category of finite groups to the category of pointed $k$-schemes such
that for each finite group $G$ there is a natural isomorphism
$\pi_1(\beta(G))\simto G$.  In this way, one can
recover the functor of points of $G$ in the category of bands over a
point, which is enough to recover $\B{G}$.  In classical language, if
$k$ is an algebraically closed field and $G$ is a finite group, the
functor $X\mapsto\H^1(X,G)$ on the category of $k$-schemes determines $G$
up to unique outer isomorphism.

The phenomena we describe strike us as pedagogically useful: even if
one is primarily concerned with the isomorphism classes of a given
moduli problem, the automorphism information in the stack is nonetheless 
already determined by the crude functorial description.

\subsection{Notation and basic definitions}
\label{sec:notat-basic-defin}

Prior to summarizing our results, we set the notation and terminology used 
throughout the paper. 

We fix a quasi-separated base scheme $S$ throughout.  (This assumption
is also hidden in \cite{l-m-b} on page x; since this is the standard
reference on the subject, we will also make blanket quasi-separation
hypotheses.)

Given a category $\mc C$ and object $T\in\mc C$, the 
slice category of $\mc C$ over $T$ will be denoted $T\textrm{-}\mc
C$.  A functor $F:\mc C^{\circ}\to\Set$ induces a functor $F|_T:(T\textrm{-}\mc C)^{\circ}\to\Set$.

We assume throughout that all stacks have small fiber categories.
Grothendieck's theory of universes can be used to see that this is a
harmless assumption in practice.

An \emph{open substack\/} of a stack $\ms X$ is an equivalence class of
morphisms $\ms U\to\ms X$ which are representable by open immersions.

\begin{defn}\label{d:notat-basic-defin-1}
  A stack is \emph{abelian\/} if it has abelian inertia stack.
\end{defn}

Following \cite{l-m-b}, all Artin stacks will be assumed to have 
quasi-compact (and hence finite type) diagonals. The notion of 
quasi-algebraic stack involves an infinitesimal condition, for which we
recall the following.

\begin{lem} Given $U_0=\Spec A_0 \to U=\Spec A$ a nilpotent closed 
immersion of affine $S$-schemes, and a morphism $U_0 \to Z$ with $Z=\Spec B$ 
another affine $S$-scheme, the pushout $U \coprod_{U_0} Z$ exists in the 
category of $S$-schemes, and is given by $\Spec(A \times_{A_0} B)$. 
\end{lem}
\begin{proof} One checks easily that the map $Z \to \Spec(A \times_{A_0} B)$
is a homeomorphism, and it is then not difficult to verify that
$\Spec(A \times_{A_0} B)$ is indeed the pushout in the category of
(possibly non-affine) schemes.
\end{proof}

\begin{defn}\label{def:quasi-algebraic}
We call an $S$-stack $\ms X$ \emph{quasi-algebraic\/} if it satisfies the
following three conditions.
\begin{enumerate}
\item The diagonal $\ms X\to\ms X\times_S\ms X$ is representable by
  separated quasi-compact morphisms of algebraic spaces.
\item The inertia stack of $\ms X$ is locally of finite presentation
  over $\ms X$.
\item Given affine $S$-schemes $U_0,U,Z$ with $U_0\inj U$ a nilpotent
  closed immersion and $U_0\to Z$ an arbitrary morphism, the induced
  functor between fiber categories
$$\ms X_{U\coprod_{U_0}X}\to\ms X_U\times_{\ms X_{U_0}}\ms X_Z$$
is an equivalence.  
\end{enumerate}
\end{defn}

\begin{remark}\label{rem:quasi-algebraic} 
  Conditions (1) and (3) are always satisfied by any Artin stack (see
  e.g.\ Lemma 1.4.4 of \cite{ol1} for the latter). Condition (2) is
  satisfied by any algebraic space, and any locally Noetherian Artin
  stack, and more generally any Artin stack for which the diagonal is
  locally of finite presentation, but not for a general Artin stack.
  Thus, our implicit hypothesis that all Artin and Deligne-Mumford
  stacks are quasi-algebraic is a non-vacuous restriction.

  For a simple example of this, consider the action of $\Z/2\Z$ on
  $k[x_1,x_2,\ldots]$ in which $1$ acts by multiplication by $-1$ on
  each variable.  The quotient stack is an Artin stack, but is not
  locally Noetherian, and does not satisfy (2). Indeed, when pulled
  back to $\Spec k[x_1,x_2,\ldots]$, the inertia stack is given by the
  extension by $0$ of the group scheme $\Z/2\Z$ supported at the
  ``origin,'' and is therefore not locally of finite presentation.

  Each of conditions (2) and (3) only arise at a single point in our
  argument, in recognizing morphisms which are locally of finite 
  presentation and smooth/\'etale respectively, 
  but these are crucial because they combine to show we can recognize
  smooth covers by schemes on the level of functors, which then leads
  to a plethora of additional recognition results.
\end{remark}

We will without further comment assume that all of our Artin 
(and Deligne-Mumford) stacks satisfy condition (2) as well, so that
they are quasi-algebraic.

\subsection{Summary of results}\label{sec:summary}

In Section \ref{sec:rec-props}, we will analyze properties of
morphisms of Artin stacks, as well as absolute properties of stacks,
showing that nearly all of the standard properties can be tested on
the level of functors. As detailed in Section
\ref{sec:props-summary}, nearly all standard
properties of morphisms and of stacks are isonatural. Most of these
results follow after we show that we can recognize smooth covers of a
quasi-algebraic stack by a scheme, although additional argument is
required for quasi-compact, separated, and proper morphisms. We also
show that we can test non-triviality of stabilizer groups from the
functor, and can in fact recover the groups whenever they are abelian.

In Section \ref{sec:isonatural-stacks}, we analyze a number of specific classes
of stacks, such as Artin stacks having an open dense substack with trivial
stabilizer, classifying stacks for finite and abelian groups, and certain
gerbes. We use specific categorical construction in each case to show
that within each class, a stack can be reconstructed from its functor,
and we then apply the results of Section \ref{sec:rec-props} to show that we
can also tell on the level of functors whether a quasi-algebraic stack
is in any of the classes in question. We thus conclude that any stack in
any of the classes is isonatural. For detailed statements, see Section 
\ref{sec:stacks-summary}.

The following theorem is a corollary of our main results, and shows
that the most common moduli stacks are in fact all isonatural.

\begin{thm}\label{thm:main-examples} Any base change of any open
  substack of any of the following stacks is isonatural.
  \begin{enumerate}
  \item The stack $\widebar{\ms M}_{g,n}$ for all $g\geq 2$
    and $n\geq 0$, as a $\Z$-stack.
  \item The Picard stack of a projective scheme flat and of finite presentation 
    which is cohomologically flat in degree $0$ over an algebraic
    space with quasi-compact connected components. 
  \item The stack of stable vector bundles on a
    projective scheme flat and of finite presentation which is
    cohomologically flat in degree $0$ over an 
    algebraic space with quasi-compact connected components.
  \item The stack of stable vector bundles of rank $n$ and fixed
    determinant on a smooth
    proper curve over any normal quasi-projective $\Z[1/n]$-scheme.
  \item The stack of stable coherent sheaves of rank $n$ with fixed
    determinant and sufficiently large discriminant on a smooth
    projective surface over any normal quasi-projective $\Z[1/n]$-scheme.
  \item The stack of $n$th roots of an invertible sheaf $\ms L$ on a
    regular algebraic space $X$.
  \end{enumerate}
\end{thm}

\section*{Acknowledgments}
\label{sec:acknowledgments}

During the course of this work we received helpful feedback and
corrections from Martin Olsson.

\section{Isonatural properties}
\label{sec:rec-props}

In this section, we prove that nearly all of the usual properties of
quasi-algebraic stacks, and of morphisms of Artin stacks, are
isonatural. We impose additional hypotheses for separateness and
properness, but otherwise our results are completely general.

\subsection{Summary of results}\label{sec:props-summary}

Suppose we are given a 1-morphism $f:\ms X \to \ms Y$ of Artin stacks
over a base scheme $S$ with induced morphism $F_f:F_{\ms X}\to F_{\ms Y}$.

We remind the reader that a \emph{trait\/} is the spectrum of a
complete discrete valuation ring.

\begin{thm}\label{thm:summary-mors-rep} The following properties of 
representable morphisms of quasi-algebraic stacks are isonatural: 
\begin{enumerate}
\item locally of finite presentation;
\item surjective;
\item smooth, assuming the source is a scheme;
\item unramified, assuming the source is a scheme;
\item \'etale, assuming the source is a scheme.
\end{enumerate}
\end{thm}

\begin{thm}\label{thm:summary-mors}
The following properties of morphisms of Artin stacks are isonatural: 
\begin{enumerate}
\item locally of finite presentation;
\item locally of finite type;
\item surjective;
\item smooth;
\item flat;
\item quasi-compact;
\item separated and locally of finite type, assuming the target is
locally Noetherian;
\item proper, assuming the target is locally Noetherian and either is
abelian or has proper inertia.
\end{enumerate}
\end{thm}

We will use these results, and in particular our ability to recognize
smooth covers by a scheme, to show further that a number of absolute
properties of stacks are isonatural.

\begin{cor}\label{cor:summary-absolute} The following properties of 
quasi-algebraic stacks are isonatural:
\begin{enumerate}
\item Artin;
\item Deligne-Mumford;
\item gerbe (over an algebraic space);
\item locally Noetherian;
\item normal;
\item reduced;
\item regular;
\item quasi-compact;
\item proper inertia.
\end{enumerate}
\end{cor}

Note that for locally Noetherian, normal, etc., we are assuming that
the stack in question is an Artin stack, because this is the only
context in which the properties are defined. However, we can also test
whether or not a stack is Artin on the level of functors.

Finally, in Section \ref{sec:nanas} we show the following.

\begin{thm}\label{T:nanas} Let $\ms X$ be a quasi-algebraic stack, and
$\eta \in \ms X_T$ for some scheme $T$. Then the following can be
recovered from the functor $F_{\ms X}$:
\begin{enumerate}
\item whether or not $\Aut(\eta)$ is abelian;
\item $\Aut(\eta)$ itself, when it is abelian.
\end{enumerate}
 
Moreover, given a morphism $T' \to T$, if $\Aut(\eta)$ and
$\Aut(\eta|_{T'}))$ are both abelian, the natural restriction map
$\Aut(\eta) \to \Aut(\eta|_{T'})$ can also be recovered from $F_{\ms X}$.
\end{thm}

Since the trivial group is abelian, it follows from Theorem
\ref{T:nanas} that triviality of $\aut(\eta)$ is also determined by
$F_{\ms X}$. We therefore immediately conclude (see Corollary
\ref{C:sheafy} for a more precise version). 

\begin{cor} Algebraic spaces are isonatural among quasi-algebraic stacks.
\end{cor} 

\subsection{Background on categories}

We begin by reminding the reader of
several basic notions from the theory of $2$-categories (by which we
mean categories enriched over groupoids).  The reader afraid of
arbitrary $2$-categories can simply think about the $2$-category of
categories: the objects are categories and the groupoid of morphisms
between two objects is the groupoid of functors (with $2$-isomorphisms
given by natural isomorphisms). In what follows, we will write $\ms C$
for a fixed $2$-category.

\begin{defn}
  A {\it $2$-commutative diagram\/} in $\ms C$ is a commutative
  diagram in the underlying $1$-category $\widebar{\ms C}$ along with
  a collection of $2$-isomorphisms determined as follows: 
  \begin{enumerate}
  \item for any two objects $A$ and $B$ in the diagram and any two
    arrows $\alpha$ and $\beta$ between $A$ and $B$ arising from
    compositions of arrows in the diagram, there is a $2$-isomorphism
    $\gamma_{\alpha,\beta}:\alpha\to\beta$;
  \item for any three paths $\alpha,\beta,\delta$ from $A$ to $B$, we
    have $\gamma_{\beta,\delta}\gamma_{\alpha,\beta}=\gamma_{\alpha,\delta}$;
  \item given two paths $\alpha,\beta:A\to B$ and a path $\delta:B\to
    C$, we have
    $\delta(\gamma_{\alpha,\beta})=\gamma_{\delta\alpha,\delta\beta}$,
    where the left-hand side arises from the composition functor on
    $\hom$-groupoids.
  \end{enumerate} 
\end{defn}

In other words, the commuting of $1$-morphisms is mediated by
$2$-morphisms, and any conceivable commutation relation among the
$2$-morphisms is assumed to be compatible.  Commutative diagrams in
this section are always assumed to be $2$-commutative; we will suppress the
$2$-morphisms from the notation.

\begin{defn}
  A \emph{$2$-Cartesian diagram\/} in $\ms C$ is a square of
  $1$-morphisms
$$\xymatrix{A\ar[r]^e\ar[d]_f & B\ar[d]^g\\
C\ar[r]_h&D}$$
along with an isomorphism $\epsilon:g\circ e\simto h\circ f$, such that for all
objects $E$ of $\ms C$, the functor
$$\Hom(E,A)\to\Hom(E,C)\times_{\Hom(E,D)}\Hom(E,B)$$
is an equivalence of groupoids.  Here the right-hand side is the usual
fibered product of groupoids (see Example \ref{ex:example} below), and
$\epsilon$ is used to compare the compositions $E\to B\to D$ and $E\to
C\to D$ coming from $E \to A$.
\end{defn}
\begin{example}\label{ex:example}
  Suppose 
$$\xymatrix{ & B\ar[d]^g\\
C\ar[r]_h & D}$$
is a pair of functors between (small) categories.  Define $B\times_D C$ to be
the category of triples $(b,c,\epsilon)$, where $b$ is an object of $B$,
$c$ is an object of $C$, and $\epsilon:g(b)\simto h(c)$ is an
isomorphism.  Then the resulting diagram 
$$\xymatrix{B\times_D C\ar[r]\ar[d] & B\ar[d]\\
C\ar[r] & D}$$
is $2$-Cartesian.  
\end{example}

\begin{lem}\label{sec:finit-prop-1} 
  Consider a $2$-commutative diagram in $\ms C$ of the form
$$\xymatrix{{A} \ar[r] \ar[d] & {B} \ar[d] \\
{A'} \ar[r] \ar[d] & {B'} \ar[d] \\
{A''} \ar[r] & {B''}.}$$
If the lower square is $2$-Cartesian then the upper square is
$2$-Cartesian if and only if the outer square is $2$-Cartesian.
\end{lem}
\begin{proof} 
Let $\epsilon, \delta, \eta$ be the isomorphisms associated to the
upper, lower, and outer squares respectively by the $2$-commutativity
of the diagram. Since the lower square is
$2$-Cartesian, the functor
$$\Hom(E,A')\to\Hom(E,B')\times_{\Hom(E,B'')}\Hom(E,A'')$$
induced by $\delta$ is an equivalence, which leads to an equivalence
$$\xymatrix{\Hom(E,B)\times_{\Hom(E,B')}\Hom(E,A')\ar[d]\\
\Hom(E,B)\times_{\Hom(E,B')}(\Hom(E,B')\times_{\Hom(E,B'')}\Hom(E,A'')).}$$
Thus, we see that the natural functor
$$\Hom(E,B)\times_{\Hom(E,B')}\Hom(E,A')\to
\Hom(E,B)\times_{\Hom(E,B'')}\Hom(E,A'')$$
induced by $\delta$ is an equivalence of groupoids. Moreover, by
$2$-commutativity we find that the natural functor
$\Hom(E,A) \to \Hom(E,B)\times_{\Hom(E,B'')}\Hom(E,A'')$ induced by
$\eta$ factors as the composition of the functor
$\Hom(E,A) \to \Hom(E,B)\times_{\Hom(E,B')}\Hom(E,A')$ induced by
$\epsilon$ with the above equivalence. It follows that the top square 
satisfies the $2$-Cartesian property if and only if the outer square does.  
\end{proof}

  We next describe a model for certain homotopy colimits
  which will arise in our study of finiteness properties.  Let 
  $\ms X\to S$ be a category fibered in groupoids over $S\Sch$.  
The following lemma is well known.

\begin{lem}\label{sec:finit-prop-4}
    There is a functorial pair $(\ms X^{\split},\sigma:\ms
    X^{\split}\to\ms X)$ consisting of a split $S$-groupoid and a
    $1$-isomorphism $\sigma$ of $S$-groupoids, such that for any $1$-morphism
    $\ms X\to\ms Y$, the diagram
$$\xymatrix{\ms X^{\split}\ar[r]\ar[d] & \ms X\ar[d]\\
\ms Y^{\split}\ar[r] & \ms Y}$$
strictly commutes (in the sense that the two compositions are equal as
functors $\ms X^{\split}\to\ms Y$).
\end{lem}

\begin{proof}
  Given $\ms X$, define $\ms X^{\split}$ as the groupoid whose fiber
  category over $T\to S$ is the groupoid of maps $\Hom_S(T,\ms X)$ (where
  $T$ denotes by abuse of notation the canonical discrete groupoid
  associated to $T$).  Composition of morphisms gives $\ms X^{\split}$
  the structure of a split groupoid (i.e., a functor from
  $S\Sch^{\circ}$ to the category
  of groupoids which satisfies the descent condition).  Moreover,
  evaluation on the identity yields a natural map $\ms
  X^{\split}\to\ms X$ which is a $1$-isomorphism since any arrow in a
  category fibered in groupoids is Cartesian. Functoriality is clear 
  from the construction.
\end{proof}

We recall the na\"ive notion of colimit in the
category of categories.  Suppose we have a filtering directed system of
categories and functors $(\mc C_i)$; here we assume that the functors
are strictly associative, so that we obtain a functor to the category of
categories.  Given also $x\in\Obj\mc C_i$, write
$x|_{\mc C_k}$ for the image of $x$ under the map $\mc C_i\to\mc C_k$
for any $k\geq i$.
 There is a colimit $\limr\mc C_i$ in the category
of categories defined as follows: the objects of $\limr\mc C_i$ are
given by the disjoint union $\coprod\Obj\mc C_i$, and, if $x\in\Obj\mc
C_i$ and $y\in\Obj\mc C_j$, the morphisms $\Hom_{\limr C_i}(x,y)$ are
given by $\limr_{k\geq i,j}\Hom_{\mc C_k}(x|_{\mc C_k},y|_{\mc C_k})$.

Now let $R_i$ be a filtering directed system of $S$-rings, and write
$R=\limr R_i$.  If $\ms X$ is an $S$-groupoid, we observe that  
$\ms X_{R_i}^{\split}$ forms a filtering direct system of categories.

\begin{defn}
  The \emph{colimit of $\ms X$ over $(R_i)$\/}, denoted $\limr\ms X_{R_i}$,
  is the na\"ive colimit $\limr \ms X_{R_i}^{\split}$.
\end{defn}

The definition makes it clear that there is a natural morphism $\limr
\ms X_{R_i}\to\ms X^{\split}_{R}$, which, composed with $\sigma_R$
yields a natural $1$-morphism $\limr\ms X_{R_i}\to\ms X_R$.  Thus,
given a $1$-morphism $\ms X\to\ms Y$, there thus results a diagram
$$\xymatrix{\limr\ms X_{R_i}\ar[r]\ar[d] & \ms X_R\ar[d]\\
\limr\ms Y_{R_i}\ar[r] & \ms Y_R}$$
which is strictly commutative.

\subsection{Finiteness properties}

We will show in Proposition \ref{P:recog-lfp} that the finiteness
hypothesis on the inertia stack of a quasi-algebraic stack implies
that for morphisms of quasi-algebraic stacks, being locally of finite
presentation is isonatural.  

\begin{defn}\label{D:loc-fin-pr} Given a scheme $S$, a morphism of 
$S$-groupoids 
$A \to B$ is \emph{locally of finite presentation\/} if, for all
filtering directed systems $R_i$ of $S$-rings, the diagram
$$\xymatrix{{\limr A_{R_i}} \ar[r] \ar[d] & {A_{\limr R_i}} \ar[d] \\
{\limr B_{R_i}} \ar[r] & {B_{\limr R_i}}}$$
is $2$-Cartesian.
\end{defn}

\begin{remark}\label{sec:finit-prop-2}
  Any functor has an associated (discrete) groupoid.  Applying
  Definition \ref{D:loc-fin-pr} in the case that $A$ and $B$ arise as
  the groupoids associated to functors yields the usual definition of
  local finite presentation for natural transformations between
  functors.  Furthermore, when $A$ and $B$ are the functors of points
  of $S$-schemes, this definition coincides with the standard
  definition for schemes, by Proposition 8.14.2 of \cite{ega43}. More
  generally, if $A$ and $B$ are Artin stacks, our definition agrees
  with the usual one, by Proposition 4.15(i) of \cite{l-m-b}.  We will
  also verify in Lemma \ref{lem:lfp-rep} below that our definition
  agrees with that of \cite{l-m-b} in the case of representable
  morphisms.
\end{remark}

From Lemma \ref{sec:finit-prop-1}, we formally conclude the following.

\begin{cor}\label{C:lpf-comp} Let $S$ be a scheme, and 
$${A} \overset{f}{\to} {B} \overset{g}{\to} {C}$$
morphisms of fibered $S$-categories, with $g$ locally of finite 
presentation. Then $f$ is locally of finite presentation if and only if 
$g \circ f$ is locally of finite presentation.
\end{cor}

\begin{lem}\label{sec:finit-prop}
  Let $A\to B$ and $C\to B$ be morphisms of $S$-groupoids.  If $A\to
  B$ is locally of finite presentation then the pullback $A\times_B
  C\to C$ is locally of finite presentation.
\end{lem}
\begin{proof}
  Let $R_i$ be a filtering directed system of $S$-rings.  Consider the
  diagram
$$\xymatrix{(A\times_B C)_{\limr R_i}\ar[rrr]\ar[ddd]&&& A_{\limr
    R_i}\ar[ddd]\\
  & \limr (A\times_B C)_{R_i}\ar[ul]\ar[r]\ar[d] & \limr
  A_{R_i}\ar[ur]\ar[d]&\\
  & \limr C_{R_i}\ar[r]\ar[dl] & \limr B_{R_i}\ar[dr] &\\
  C_{\limr R_i}\ar[rrr]&&& B_{\limr R_i}.\\}$$ The right ``square'' of
the diagram is $2$-Cartesian by hypothesis and the outer square is
$2$-Cartesian by definition of the fiber product.  The inner square is
$2$-Cartesian since colimits of $2$-Cartesian squares are
$2$-Cartesian.  Applying Lemma \ref{sec:finit-prop-1}, we see that the
left square is $2$-Cartesian, as desired.
\end{proof}

Definition 3.10.1 of \cite{l-m-b} defines a representable morphism of
stacks $\ms X\to\ms Y$ to be locally of finite presentation if
for all $U\to\ms Y$, the fiber product $\ms X\times_{\ms X}U\to U$ is
locally of finite presentation.  As we show in the following lemma,
this is actually no different from our definition above.

\begin{lem}\label{lem:lfp-rep} A representable morphism $f:\ms X \to \ms Y$ 
of stacks is locally of finite presentation if and only if for all
schemes $U\to\ms Y$, the algebraic space $\ms X\times_{\ms Y}U\to U$ is 
locally of finite presentation.
\end{lem}

\begin{proof} If $f$ is locally of finite presentation and 
  $U\to\ms Y$ is an object, then the morphism $\ms X\times_{\ms Y}U\to
  U$ is locally of finite presentation by Lemma \ref{sec:finit-prop}
  and Remark \ref{sec:finit-prop-2}.

Now assume that $f:\ms X\times_{\ms Y}U\to U$ is locally of finite
presentation for all objects $U\to\ms Y$, and let $R_i$ be a filtering
directed system of $S$-rings.  By 
Proposition I.8.1.6 of \cite{sga41}, there is a cofinal subsystem
whose indexing category is a partially ordered set with an initial
element.  This yields $R_0$ in the system which maps uniquely to every $R_i$ 
so that the system becomes a system of $R_0$-algebras.

Let $\alpha$ be an object of $\ms X_{\limr R_i}$ with image $\beta$ in
$\ms Y_{\limr R_i}$.  Moreover, suppose $b\in\ms Y_{R_1}$ is an object
and $\phi:b\to\beta$ is an isomorphism in $\ms Y_{\limr R_i}$.  We may
clearly suppose that $R_1=R_0$.  We wish to show that the triple
$(\alpha,b,\phi)$ arises from an object $\gamma_i$ of $\ms X_{R_i}$ for
some $i$, and that any two such objects $\gamma_i$ and $\gamma_j$
become uniquely isomorphic in $\ms X_{R_k}$ for some $k$ larger than
$i$ and $j$.  

Let $U=\Spec R_0$.  The system $R_i$ becomes a system of $U$-rings,
and the relative finite presentation condition tells us that the
diagram
\begin{equation*}
  \label{eq:diag}
  \xymatrix{\limr (\ms X\times_{\ms Y} U)_{R_i}\ar[r]\ar[d]& (\ms
  X\times_{\ms Y} U)_{\limr R_i}\ar[d]\\
\limr U_{R_i}\ar[r] & U_{\limr R_i}}
\end{equation*}
is $2$-Cartesian.  Using the compatibility of colimits with fibered
product, we can rewrite the diagram as 
\begin{equation}
  \label{eq:diag2}
\xymatrix{\limr \ms X_{R_i}\times_{\limr \ms Y_{R_i}}\limr U_{R_i}\ar[r]\ar[d]
& \ms X_{\limr R_i}\times_{\ms Y_{\limr R_i}}U_{\limr R_i}\ar[d]\\
\limr U_{R_i}\ar[r] & U_{\limr R_i}.}  
\end{equation}
The system yields a canonical element $a\in
U_{\limr R_i}$.  The triple $(\alpha,b,\phi)$ yields an object of the
upper right category of diagram \eqref{eq:diag2} mapping to $a$.  On
the other hand, $a$ is the image of any of the structure morphisms
occuring in $\limr U_{R_i}$.  Since \eqref{eq:diag2} is $2$-Cartesian,
we see that there is a $\gamma_i$ mapping to $(\alpha,b,\phi)$, and
that any two $\gamma_i$ and $\gamma_j$ become uniquely isomorphic in
$\ms X_{R_k}$ for large enough $k$, as required.  
\end{proof}

\begin{notn} Given a (small) category $C$, we denote by $F_C$ the set of 
isomorphism classes of objects in $C$. We also denote by $I_C$ the
category whose objects are pairs $(\eta,\varphi)$ with $\eta \in C$,
$\varphi \in \Aut(C)$, and whose morphisms are morphisms $f:\eta \to \eta'$ 
in $C$ with $f \circ \varphi=\varphi' \circ f$.
\end{notn}

\begin{lem}\label{sec:finit-prop-3} 
Given a directed system of categories $A_i$, the
natural map $\limr F_{A_i} \to F_{\limr A_i}$ is a bijection.
The natural morphism $\limr I_{A_i} \to I_{\limr A_i}$ is an equivalence.
\end{lem}

\begin{proof} The first assertion follows easily from the definition of a 
colimit of categories. The second is slightly more involved, but still routine.
\end{proof}

\begin{prop}\label{P:func-lpf} For any quasi-algebraic $S$-stack $\ms X$, 
the natural map $\ms X \to F_{\ms X}$ is locally of finite presentation.
\end{prop}

\begin{proof} 
  We will show that this follows from the hypothesis that the inertia
  stack of $\ms X$ is finitely presented over $\ms X$.  We need to see
  that for all filtering directed systems $R_i$ of $S$-rings, the
  diagram
$$\xymatrix{{\limr \ms X_{R_i}} \ar[r] \ar[d] & {\ms X_{\limr R_i}} \ar[d] \\
  {\limr F_{\ms X}(R_i)} \ar[r] & {F_{\ms X}(\limr R_i)}}$$ is
$2$-Cartesian. By the first part of Lemma \ref{sec:finit-prop-3}, we
may replace $\limr F_{\ms X}(R_i)$ by $F_{\limr \ms X_{R_i}}$.  We
then need to check that the morphism 
$$n:\limr \ms X_{R_i}\to \ms X_{\limr R_i}
\times_{F_{\ms X}(\limr R_i)} F_{\limr \ms X_{R_i}}$$
 is essentially surjective and fully faithful.

An object of $\ms X_{\limr R_i} \times_{F_{\ms X}(\limr R_i)} 
F_{\limr \ms X_{R_i}}$ is a pair $(\eta,\bar{\mu})$ with 
$\eta \in \ms X_{\limr R_i}$ and $\bar{\mu} \in F_{\limr \ms X_{R_i}}$, 
both mapping to the same element in $F_{\ms X}(\limr R_i)$. But if
$\mu \in \limr \ms X_{R_i}$ is an object representing $\bar{\mu}$, we check
easily that that $\mu$ maps to an object isomorphic to $(\eta,\bar{\mu})$,
proving essential surjectivity.

We claim that the full faithfulness of $n$ follows from the hypothesis
that the inertia stack of $\ms X$ is locally of finite presentation
over $\ms X$. We note that a morphism $(\eta,\bar{\mu}) \to
(\eta',\bar{\mu}')$ in $\ms X_{\limr R_i} \times_{F_{\ms X}(\limr
  R_i)} F_{\limr \ms X_{R_i}}$ is simply a morphism $\eta \to \eta'$,
together with the requirement that $\bar{\mu}=\bar{\mu}'$. If
$(\eta,\bar{\mu})$ and $(\eta',\bar{\mu}')$ are the images of
$\mu,\mu' \in \limr \ms X_{R_i}$, we therefore need to check that
under the hypothesis that $\mu,\mu'$ are isomorphic, the
(iso)morphisms $\mu \to \mu'$ are in bijection with (iso)morphisms
$\eta \to \eta'$. Fixing a choice of isomorphism $\mu \to \mu'$, it is
therefore enough to see that $\Aut(\mu)$ is in bijection with
$\Aut(\eta)$.

This last assertion is precisely what is given by the hypothesis that
the inertia stack of $\ms X$ is locally of finite presentation over
$\ms X$: in the $2$-Cartesian diagram
$$\xymatrix{{\limr \ms I(\ms X)_{R_i}} \ar[r] \ar[d] & {\ms I(\ms X)_{\limr R_i}} 
\ar[d] \\
{\limr \ms X_{R_i}} \ar[r] & {\ms X_{\limr R_i}},}$$ 
after applying the second part of Lemma \ref{sec:finit-prop-3} to replace
$\limr \ms I(\ms X)_{R_i}$ by $I_{\limr \ms X_{R_i}}$,
the essential surjectivity of the map to the $2$-fiber product implies
that $\Aut(\mu) \to \Aut(\eta)$ is surjective, while the full faithfulness
implies injectivity.
\end{proof}

\begin{remark} If we had imposed the stronger condition that the diagonal
of $\ms X$ is locally of finite presentation, we would have that the
map
$\limr \ms X_{R_i} \to \ms X_{\limr R_i}$ is fully faithful. This is not
clearly true under our weaker hypothesis.
\end{remark}

\begin{prop}\label{P:recog-lfp} For morphisms of quasi-algebraic stacks, 
the property of being locally of finite presentation is isonatural. 
Specifically, a morphism $f:\ms X \to \ms Y$ of quasi-algebraic stacks 
is locally of finite presentation if and only if the induced morphism
$F_{\ms X} \to F_{\ms Y}$ is locally of finite presentation.
\end{prop}

\begin{proof} First suppose that $F_{\ms X} \to F_{\ms Y}$ is locally
of finite presentation. By Proposition \ref{P:func-lpf}, 
$\ms X \to F_{\ms X}$ and $\ms Y \to F_{\ms Y}$ are both locally
of finite presentation, so applying Corollary \ref{C:lpf-comp} twice
we see that $\ms X \to F_{\ms Y}$ and thus $\ms X \to \ms Y$ are
locally of finite presentation.

Conversely, suppose that $f$ is locally of finite presentation. We thus
have that
$$\xymatrix{{\limr \ms X_{R_i}} \ar[r] \ar[d] & {\ms X_{\limr R_i}} \ar[d] \\
{\limr \ms Y_{R_i}} \ar[r] & {\ms Y_{\limr R_i}}}$$
is 2-Cartesian and we wish to see that
$$\xymatrix{{\limr F_{\ms X}(R_i)} \ar[r] \ar[d] & {F_{\ms X}(\limr R_i)}
\ar[d] \\
  {\limr F_{\ms Y}(R_i)} \ar[r] & {F_{\ms Y}(\limr R_i)}}$$ 
is (2-)Cartesian. Moreover, by Lemma \ref{sec:finit-prop-3}, the
latter diagram is (2-)Cartesian if and only if 
$$\xymatrix{{F_{\limr \ms X_{R_i}}} \ar[r] \ar[d] & {F_{\ms X}(\limr R_i)}
\ar[d] \\
  {F_{\limr \ms Y_{R_i}}} \ar[r] & {F_{\ms Y}(\limr R_i)}}$$ 
is (2-)Cartesian; that is, if and only if the first diagram remains 
(2-)Cartesian after passing to isomorphism classes. Although this
is not true for arbitrary groupoids, it will be true under the
quasi-algebraicity hypothesis. The only obstruction that could arise
would be automorphisms of objects in $\ms Y_{\limr R_i}$ which do not
lift to automorphisms of either $\limr \ms Y_{R_i}$ or $\ms X_{\limr R_i}$.
But we see that the hypothesis that the inertia stack of $\ms Y$ is 
locally of finite presentation over $\ms Y$ tells us precisely that
every automorphism of an object over $\ms Y_{\limr R_i}$ lifts to an
automorphism over $\limr \ms Y_{R_i}$, giving us the desired result.
\end{proof}

\subsection{Formal criteria}\label{sec:formal-crit}

We next show that under the mild deformation-theoretic hypotheses 
of quasi-algebraic stacks, the formal criteria for 
smoothness, unramifiedness, and \'etaleness can be rephrased
functorially. 
It then follows that we can test for smooth-local properties of 
morphisms and of Artin stacks. We begin with a general remark on
the sort of $2$-commutative diagrams arising in formal and valuative
criteria.

\begin{remark}\label{rem:2-diag-objs} For a stack $\ms Y$ and a scheme $T'$,
a morphism $T' \to \ms Y$ is equivalent to an object $\eta \in \ms Y_{T'}$
together with a choice of pullback $i^*(\eta) \in \ms Y_T$ for all
scheme morphisms $i: T \to T'$. Fix a morphism of stacks 
$f: \ms X \to \ms Y$, and a morphism $i:T \to T'$ of schemes. 
A $2$-commutative diagram 
$$\xymatrix{{T} \ar[r] \ar^{i}[d] & {\ms X}\ar^{f}[d] \\
{T'} \ar[r] & {\ms Y}}$$
yields objects $\mu \in \ms X_T$ and $\eta \in \ms Y_{T'}$, a choice
of pullback $i^* \eta \in \ms Y_T$, and an isomorphism 
$\gamma:i^* \eta \risom f(\mu)$. Conversely, given $(\mu,\eta,\gamma)$,
and choices of arbitrary pullbacks for $\mu$ and for $\eta$ yielding
morphisms $T \to \ms X$ and $T' \to \ms Y$, we find that $\gamma$ is
precisely the data of a $2$-isomorphism determining a $2$-commutative
diagram.

Next, a morphism $j:T' \to \ms X$
yields an object $\mu' \in \ms X_{T'}$ and a choice of pullback
$i^* \mu' \in \ms X_T$, and gives rise to a $2$-commutative diagram
$$\xymatrix{{T} \ar[r] \ar^{i}[d] & {\ms X}\ar^{f}[d] \\
{T'} \ar[r]\ar^{j}[ur] & {\ms Y}}$$
if and only if there exist isomorphisms $\alpha:i^* \mu' \risom \mu$ 
and $\beta:\eta \to f(\mu')$ such that $\gamma=i^*(\beta) \circ f(\alpha)$.
Note here that the last condition makes sense because $j$ also induces a 
map $T' \overset{f \circ j}{\to} \ms Y$, so we obtain also a choice of 
$i^* f(\mu')$, and see that we have $i^* f(\mu')=f(i^* \mu')$.

We therefore see that with $f$ and $i$ fixed, the statement that for every
$2$-commutative diagram 
$$\xymatrix{{T} \ar[r] \ar^{i}[d] & {\ms X}\ar^{f}[d] \\
{T'} \ar[r] & {\ms Y}}$$
there exists a morphism $T' \to \ms X$ and isomorphisms making the
new diagram $2$-commutative is equivalent to the statement that all
triples $(\mu,\eta,\gamma)$ as above, there exist $\mu', \alpha, \beta$
as above with $\gamma=i^*(\beta) \circ f(\alpha)$. In addition, if
$\ms X$ has no non-trivial automorphisms, we see similarly that the
uniqueness of a morphism $T' \to \ms X$ and isomorphisms making the
new diagram $2$-commutative is equivalent to uniqueness of the
$\mu'$ such that there exists $\alpha,\beta$ with
$\gamma=i^*(\beta) \circ f(\alpha)$.
\end{remark}

\begin{defn}\label{def:formal-crits}
  A morphism $\ms X\to\ms Y$ is \emph{formally smooth\/} (resp.\
  \emph{formally \'etale, formally unramified\/}) if for every
  nilpotent closed immersion $U_0\inj U$, with $U$ the spectrum of a
  strictly Henselian local ring, every $2$-commutative diagram  
$$\xymatrix{U_0\ar[r]\ar[d] & \ms X\ar[d]\\
U\ar[r] & \ms Y}$$
extends to a (resp.\ to exactly one, resp.\ to at most one) $2$-commutative
diagram
$$\xymatrix{U_0\ar[r]\ar[d] & \ms X\ar[d]\\
U\ar[r]\ar@{-->}[ur] & \ms Y.}$$
\end{defn}

The requirement that the diagrams be $2$-commutative makes it
transparent that these conditions are compatible with base change in
$\ms Y$.  Thus, if $\ms X\to\ms Y$ is representable, this gives a
variant of the usual formal criterion for morphisms of algebraic
spaces, with the only difference being the restriction to the strictly
Henselian local case.  As in the proof of Proposition 4.15(ii) of \cite{l-m-b}, when
$\ms X\to\ms Y$ is locally of finite presentation, this variant
suffices to establish that $\ms X\to\ms Y$ is smooth (resp.\ \'etale,
resp.\ unramified).  (We remind the reader that e.g., smoothness of a
morphism $f$ is by definition equivalent to the formal criterion of
smoothness for $f$ combined with the local finite presentation of $f$,
as in Definition 17.3.1 of \cite{ega44}.  This holds true for
algebraic spaces and Artin stacks; the usual local descriptions of
such morphisms then follow from the compatibility of the conditions
with various kinds of diagrams and the classical results for schemes.)

\begin{prop}\label{P:formal-crit-affine} Let $X$ be an affine scheme, and 
$\ms Y$ a quasi-algebraic
stack, and suppose we are given a morphism $f:X \to \ms Y$, locally
of finite presentation. Then $f$ is smooth (respectively, unramified, 
\'etale)
if and only if for every nilpotent closed immersion $U_0 \to U$ of strictly
Henselian local affine schemes, every morphism $U_0 \to X$, and every object 
$\eta \in F_{\ms Y}(U \coprod_{U_0} X)$ pulling back to 
$f \in F_{\ms Y}(X)$, there exists (respectively, there is at most one,
there exists a unique) $\mu:U \to X$ such that 
$(\id \coprod \mu)^* \eta=\delta^* p^{1*} \eta$ in 
the diagram

$$\xymatrix{ F_{\ms Y}(U\coprod_{U_0} X)\ar[d]_-{p^{1\ast}}\ar[dr]^-{(\id\coprod\mu)^{\ast}} & \\
F_{\ms Y}(U)\ar[r]_-{\delta^{\ast}} & F_{\ms Y}(U\coprod_{U_0} U).}$$

Here $\delta: U \coprod_{U_0} U \to U$ is the codiagonal, and 
$p^1:U \to U \coprod_{U_0} X$ is the first inclusion.
\end{prop}

Notice that because $X$ is a scheme and $\ms Y$ is quasi-algebraic,
$f$ is necessarily representable, so smooth, unramified and \'etale
are well-defined properties.

\begin{proof} The key assertions are that morphisms $U_0 \to U$ and 
$U_0 \to X$, together with objects $\eta$ as above, are equivalent in
the sense of Remark \ref{rem:2-diag-objs} to $2$-commutative diagrams
$$\xymatrix{{U_0}\ar[r]\ar[d] & {X} \ar^{f}[d] \\
{U} \ar[r] & {\ms Y}}$$
as in the formal criteria of Definition \ref{def:formal-crits}, 
and that a map $\mu:U \to X$ makes a $2$-commutative diagram if and only
if $(\id \coprod \mu)^* \eta=\delta^* p^{1*} \eta$. Given these 
assertions, the proposition follows immediately from the standard formal
criteria in the context of representable morphisms of stacks.

Checking these assertions relies on the fact that because of the 
quasi-algebraicity hypothesis, $\ms Y_{U \coprod_{U_0} X}$ is equivalent 
to $\ms Y_U \times_{\ms Y_{U_0}} \ms Y_X$, so that 
$F_{\ms Y}(U \coprod_{U_0} X)$ is described by triples $(\zeta, f, \varphi)$
where $\zeta \in \ms Y_U$, $f \in \ms Y_X$, and 
$\varphi: \zeta|_{U_0} \risom f|_{U_0}$. Two such triples are equivalent
if there are isomorphisms of $\zeta$ and of $f$ commuting with $\varphi$
(abusing notation slightly, we henceforth consider $f$ to be an object
of $\ms Y_X$). Since we have fixed $f$ in advance, we see that our 
$\eta \in F_{\ms Y}(U \coprod_{U_0} X)$ is equivalent
to $\zeta \in \ms Y_U$ together with $\varphi:\zeta|_{U_0} \risom f|_{U_0}$,
which together with the maps $U_0 \to U$ and $U_0 \to X$ corresponds to the
data of a $2$-commutative diagram as above, as asserted.

Similarly, $F_{\ms Y}(U \coprod_{U_0} U)$ consists of pairs of objects
$\zeta,\zeta' \in \ms Y_U$ together with an isomorphism 
$\varphi':\zeta|_{U_0} \risom \zeta'|_{U_0}$, once again up to pairs of
isomorphisms commuting with $\varphi'$. Writing $\eta=(\zeta,f,\varphi)$ as 
above, we have that $(\id \coprod \mu)^* \eta$ consists of the pair 
$\zeta,\mu^* f$ glued via $\varphi$ and the canonical isomorphism
$(\mu^* f)|_{U_0} \risom f|_{U_0}$. On the other hand, 
$\delta^* p^{1*} \eta$ consists of $\zeta$ glued to itself along the 
identity. For these to be isomorphic means we have isomorphisms
$\alpha_1:\zeta \risom \zeta$ and $\alpha_2: \zeta \risom \mu^* f$
such that $\varphi \circ \alpha_1|_{U_0}=\alpha_2|_{U_0} \circ \id = 
\alpha_2|_{U_0}$, or equivalently, such that 
$\varphi = (\alpha_2 \circ \alpha_1^{-1})|_{U_0}$, and we see that this is
precisely equivalent to the $2$-commutativity of the diagram produced by 
adding $\mu$.
\end{proof}

In order to use Proposition \ref{P:formal-crit-affine} to detect the
presence of a smooth cover, we also need the following straightforward
result about the isonaturality of surjectivity.

\begin{lem}\label{L:recog-surj} A representable morphism $f:\ms X \to
  \ms Y$ of stacks is surjective if and only if the associated map
of functors is surjective on geometric points. 
\end{lem}

We can now prove all our desired results on representable morphisms.

\begin{proof}[Proof of Theorem \ref{thm:summary-mors-rep}] The
assertions on locally of finite presentation and surjectivity are
Lemma \ref{lem:lfp-rep} and Lemma \ref{L:recog-surj}, respectively.
Suppose $\ms X$ is quasi-algebraic and $U$ is a scheme. We need to
show that smoothness, unramifiedness, and 
\'etaleness of a morphism $U \to \ms X$ are all isonatural. However,
$U \to \ms X$ locally of finite presentation is smooth (respectively, 
unramified, \'etale) if and only if there exists an affine cover $\{U_i\}$ of 
$U$ such that each $U_i \to \ms X$ is smooth (respectively,
unramified, \'etale), and these properties are isonatural
for $U_i \to \ms X$ by Proposition \ref{P:formal-crit-affine}.
\end{proof}

We can now conclude the following.

\begin{cor}\label{C:recog-covers} 
  Given a quasi-algebraic stack $\ms X$, the existence of a scheme $U$
  and a smooth (respectively, \'etale) surjection $U\to\ms X$ is isonatural. 

In particular, the property of being an Artin or Deligne-Mumford
stack is isonatural.

  Moreover, the smoothness of a morphism $\ms X\to\ms Y$ of Artin
  stacks is isonatural.
\end{cor}

\begin{proof} The first assertion follows immediately from Theorem 
\ref{thm:summary-mors-rep}. We then conclude that being Artin is isonatural
because a quasi-algebraic stack is an Artin stack if and only if it has
a smooth cover by a scheme. Similarly, being Deligne-Mumford is isonatural
because a quasi-algebraic stack is Deligne-Mumford if and only if it has 
an \'etale cover by a scheme; see Definition 4.1 of \cite{l-m-b}, noting that 
the algebraic space $X$ of {\it loc.\/ cit.\/} can be replaced by a scheme by 
Definition 1.1 of \cite{l-m-b}.

Next, suppose that $f:\ms X \to \ms Y$ is a morphism of Artin
stacks. We note that $f$ is smooth if and only if there exists a smooth
cover $U \to \ms X$ by a scheme $U$ such that $U \to \ms Y$ is smooth,
but we just saw that both of these properties are isonatural, showing
that smoothness of $f$ is isonatural.
\end{proof}

This immediately allows us to finish proving isonaturality of nearly 
all properties of quasi-algebraic stacks.

\begin{cor}\label{cor:smooth-local-abs} If $P$ is a property of
Artin stacks which is local for the smooth topology then $P$ is
isonatural. Furthermore, quasi-compactness of Artin stacks is
isonatural.
\end{cor}

\begin{proof} The first assertion is trivial from Corollary 
\ref{C:recog-covers}. Next, although quasi-compactness is not
smooth local, it is defined (Definition 4.7.2 of \cite{l-m-b}) in terms 
of the existence of a smooth cover by a quasi-compact scheme, so it is 
likewise isonatural.
\end{proof}

It thus follows that every property of Artin stacks listed on p.\ 31 of 
\cite{l-m-b} is isonatural. We can also use isonaturality of smooth covers 
to recognize a range of properties of morphisms of Artin stacks.

\begin{lem}\label{L:mors-local} Let $P$ be a property of morphisms of 
schemes such that smooth covers have $P$, and in fact $P$ is local in the 
smooth topology (as in p.\ 33 of \cite{l-m-b}), and further assume that $P$ 
is stable under composition and base extension.

Then a morphism $f:\ms X \to \ms Y$ of Artin stacks has $P$ if and only
if there exist smooth covers
$T' \to {\ms X}$, $T \to {\ms Y}$ such that $T' \to {\ms Y}$ factors
through $T$, with the map $T' \to T$ having $P$. 
\end{lem}

\begin{proof} If $f$ has $P$, we let $T \to {\ms Y}$ be any smooth cover, 
and $T'$ any smooth cover of the fiber product $T \times_{\ms Y} \ms X$, 
and by definition (Definition 4.14 of \cite{l-m-b}), the map $T' \to T$ will 
have $P$.

Conversely, suppose the covers $T,T'$ exist. Note that $T' \to \ms Y$ is
then a composition of morphisms having $P$, so has $P$. We then check that 
$T \times _{\ms Y} T'$ is a smooth cover of $T' \times_{\ms Y} \ms X$,
and the natural map $T \times _{\ms Y} T' \to T$ has $P$, so by definition,
we conclude that $f$ has $P$.
\end{proof}

We immediately conclude the following from the lemma and Corollary 
\ref{C:recog-covers}.

\begin{cor}\label{C:kitchen-sink} Any property $P$ of a morphism of 
schemes satisfying the hypotheses of Lemma \ref{L:mors-local} is isonatural 
as a property of morphisms of Artin stacks. In particular, flat, surjective, 
and locally of finite type are each isonatural for morphisms of Artin stacks.
\end{cor}

Indeed, it follows that every property of morphisms listed on p.\ 33 
of \cite{l-m-b} is isonatural. 

 \begin{remark}\label{R:formal-criteria} 
   The difference between the stack-theoretic and functor-theoretic
   formal criteria is easiest to see in the context of the criterion
   for unramifiedness. Here, neither version implies the other. For
   instance, if we work over an algebraically closed field $k$, the
   map $\Spec k \to \B{G_m}$ is ramified, but appears unramified on
   the level of functors, as there is no non-trivial $G_m$-torsor over
   any local scheme. On the other hand, if we take the natural
   quotient map $\A^1_k \to [\A^1_k/\m_n]$ for $n$ prime to the
   characteristic of $k$, we have an unramified map of stacks which
   appears to be ramified at the origin on the level of functors,
   since $n$ tangent vectors all map to the same isomorphism class of
   $[\A^1_k/\m_n]_{k[\epsilon]/\epsilon^2}$.

   In contrast, it is easy to check that the stack-theoretic formal
   criterion for smoothness implies the functor-theoretic version.  On
   the other hand, one can also check that for maps of the form $\Spec
   k \to \B{G}$, the functor-theoretic formal criterion for smoothness
   does imply the stack-theoretic version. It is not clear how
   generally this equivalence might hold.
\end{remark}

\subsection{Quasi-compactness}
\label{sec:quasi-compactness}

The goal of this section is to prove the following.

\begin{prop}\label{P:qc-funcy} 
  Given a morphism $f:\ms X \to \ms Y$ of Artin stacks, the property
  that $f$ is quasi-compact is isonatural.
\end{prop}

We begin with three general lemmas.

\begin{lem}\label{L:open-subs} 
  Given a stack $\ms Y$, the functor $F$ induces a bijection between
  open substacks $\ms U\subset\ms Y$ and open subfunctors $F_{\ms
    U}\subset F_{\ms Y}$.
\end{lem}

\begin{proof} 
  Given an open substack, the associated map of functors gives an open
  subfunctor, by definition. On the other hand, given an open
  subfunctor $F'$ of $F_{\ms Y}$, the fiber product $F' \times
  _{F_{\ms Y}} \ms Y$ is an open substack of $\ms Y$.  This gives the
  desired correspondence.
\end{proof}

\begin{lem}\label{sec:quasi-compactness-1}
  Every Artin stack $\ms Y$ has a cover by open, quasi-compact
  substacks.
\end{lem}
\begin{proof} 
  Let $Y \to \ms Y$ be a smooth cover of $\ms Y$ by a scheme, and let
  $\{U_i\}$ be an affine cover of $Y$. The images $\{\ms U_i\}$ of the
  $\{U_i\}$ in $\ms Y$ are open and quasi-compact: indeed, for a
  representable morphism of stacks, one can define an image subfunctor
  in terms of $T$-valued points, which for a smooth morphism will be
  an open subfunctor; by Lemma \ref{L:open-subs} we obtain open
  substacks $\ms U_i$ with smooth covers by the $U_i$, which then
  implies also that they are quasi-compact (see Definition 4.7.2 of
  \cite{l-m-b}).
\end{proof}

  Recall that the definition of quasi-compactness (see Definition 4.16
  of \cite{l-m-b}) is that for every $Y \to \ms Y$ with $Y$ an affine
  scheme over $S$, the fiber product $\ms X \times_{\ms Y} Y$ is a
  quasi-compact stack.

\begin{lem}\label{sec:quasi-compactness-2}
  If $f: \ms X \to \ms Y$ is a quasi-compact morphism of Artin stacks,
  and $\ms Y$ is quasi-compact, then $\ms X$ is also quasi-compact.
\end{lem}
\begin{proof} 
  Let $Y \to \ms Y$ be a smooth cover by a quasi-compact scheme, and
  $\{Y_i\}$ a finite open affine cover of $Y$. By the definition
  of quasi-compact morphism, $Y_i \times_{\ms Y} \ms X$ has a smooth
  cover by a quasi-compact scheme $X_i$. The disjoint union of
  the $X_i$ then gives a quasi-compact smooth cover of $\ms X$.
\end{proof}

\begin{lem}\label{sec:quasi-compactness-3} 
  A morphism $f: \ms X \to \ms Y$ of Artin stacks is quasi-compact if
  and only if for every open quasi-compact substack $\ms Y'$ of $\ms
  Y$, the fiber product $\ms X \times _{\ms Y} \ms Y'$ is
  quasi-compact.
\end{lem}
\begin{proof} 
  First suppose that $f$ is quasi-compact, and we are given $\ms Y'$ a
  quasi-compact open substack of $\ms Y$. Then by Lemma
  \ref{sec:quasi-compactness-1}, $\ms Y'$ has a smooth cover $Y \to
  \ms Y'$ by a quasi-compact scheme, which we can assume without loss
  of generality to be affine. Since $f$ is quasi-compact, we have that
  $\ms X \times_{\ms Y} Y$ is quasi-compact, and is a smooth cover of
  $\ms X \times_{\ms Y} \ms Y'$, so we conclude that the latter is
  quasi-compact.

  Conversely, suppose our condition is satisfied, and suppose we have
  $Y \to \ms Y$, with $Y$ an affine scheme over $S$. Taking a cover of
  $\ms Y$ by open quasi-compact substacks $\ms Y_i$, the preimages
  $Y_i$ of $\ms Y_i$ in $Y$ form an open cover.  For each $i$, by
  hypothesis we have $\ms X \times_{\ms Y} \ms Y'_i$ quasi-compact, so
  it follows by Remark 4.17(1) of \cite{l-m-b} that $\ms X \times
  _{\ms Y}\ms Y_i' \to\ms Y_i'$ is quasi-compact. Taking the base
  change to $Y$, we conclude that $\ms X \times _{\ms Y} Y_i \to
  Y_i$ is quasi-compact for each $i$, and therefore that $\ms X \times
  _{\ms Y} Y \to Y$ is quasi-compact. Thus $\ms X \times _{\ms Y} Y$
  is quasi-compact, as desired.
\end{proof}

Our last lemma is that associated functors commute with fiber products
when one morphism is an open immersion.

\begin{lem}\label{sec:quasi-compactness-5} 
  If $f:\ms X \to \ms Y$ is a morphism of stacks, and $i:\ms Y' \to
  \ms Y$ an open immersion, then the natural map
$$F_{\ms X \times _{\ms Y} \ms Y'} \to 
F_{\ms X} \times _{F_{\ms Y}} F_{\ms Y'}$$ is an isomorphism of
functors.
\end{lem}
\begin{proof} 
  Indeed, a $T$-object of $\ms X \times _{\ms Y} \ms Y'$ consists of
  $T$-objects $\eta_{\ms X}$ and $\eta_{\ms Y'}$ of $\ms X$ and $\ms
  Y'$, together with an isomorphism $\eta_{\ms X}|_{\ms Y} \risom
  \eta_{\ms Y'}|_{\ms Y}$. In general, one could have two such objects
  glued by two different isomorphisms which are not related by
  automorphisms of $\ms X$ and $\ms Y'$. However, when $\ms Y'$ is an
  open substack of $\ms Y$, the natural map $\Aut(\eta_{\ms Y'}) \to
  \Aut(\eta_{\ms Y'}|_{\ms Y})$ is surjective, so this does not occur.
  Therefore, when we pass to isomorphism classes, we get the desired
  isomorphism of functors.
\end{proof}

Finally, we can prove Proposition \ref{P:qc-funcy}.

\begin{proof}[Proof of Proposition \ref{P:qc-funcy}] 
  By Lemma \ref{sec:quasi-compactness-3},
  $f:\ms X \to \ms Y$ is quasi-compact if and only if for all $\ms Y'$
  quasi-compact open substacks of $\ms Y$, we have 
  $\ms X \times_{\ms Y} \ms Y'$ quasi-compact. We know that open substacks 
  are in natural correspondence with open subfunctors by Lemma 
  \ref{L:open-subs} and that we can recognize when an Artin stack is 
  quasi-compact from its functor by Corollary \ref{cor:smooth-local-abs}.  
  Finally, we can recover $F_{\ms X \times _{\ms Y} \ms Y'}$ from
  $F_{\ms X}, F_{\ms Y}, F_{\ms Y'}$ from Lemma 
  \ref{sec:quasi-compactness-5}. We thus conclude that quasi-compactness
  of $f$ is isonatural, as desired.
\end{proof}

\subsection{Functorial valuative criteria}

We conclude our tour of properties of morphisms by addressing 
separatedness and properness, modifying the valuative 
criteria slightly to obtain criteria in terms of $F_f$. 

\begin{defn} 
A {\it doubled trait\/} is the non-separated scheme obtained by gluing
a trait to itself along the generic point.
\end{defn}

Given a trait $Q$, we let $T_Q$ denote the doubled trait associated to
$Q$.  There is a natural morphism $\chi:T_Q\to Q$.  Given a stack $\ms
X$, we will call an element $\eta\in F_{\ms X}(T_Q)$ \emph{constant\/}
if it has the form $\chi^{\ast}\eta'$ for some $\eta'\in F_{\ms
  X}(Q)$.  When $Q$ is implicit, it will be omitted from the notation.
 
\begin{lem}\label{L:separated-funcy}  
  Let $f:\ms X \to \ms Y$ be a morphism of Artin stacks, locally of 
  finite type, with $\ms Y$ locally Noetherian. Then separatedness of
  $f$ is isonatural.  Specifically, $f$ is separated if and only if 
  for every doubled trait $T$, an object of $F_{\ms X}(T)$ is constant
  if and only if its image in $F_{\ms Y}(T)$ is constant.
\end{lem}
 
\begin{proof}   
  We show that this is equivalent to the valuative criterion for  
  separatedness (Proposition 7.8 of \cite{l-m-b}). Let $T_1$ and $T_2$
  denote the two traits (canonically identified with $Q$) glued to 
  obtain $T$.  Given a stack $\ms Z$, consider the natural map $F_{\ms
    Z}(T)\to F_{\ms Z}(T_1)\times_{F_{\ms Z}(U)}F_{\ms Z}(T_2)$. 
  Given objects $\alpha_i\in\ms Z_{T_i}$, $i=1,2$, with isomorphism 
  classes $\widebar{\alpha_i}$, there is a bijection between the fiber
  of $F_{\ms Z}(T)$ over $(\widebar{\alpha}_1,\widebar{\alpha}_2)$ and
  the double coset space 
  $\aut(\alpha_2)\backslash\isom(\alpha_1|_U,\alpha_2|_U)/\aut(\alpha_1)$. (We
  can identify $\aut(\alpha_i)$ with a subgroup of $\aut(\alpha_i|_U)$
  because the diagonal of an Artin stack is assumed separated by 
  definition.)  There is a distinguished double pseudo-coset $\ast$ 
  given by the subset $\isom(\alpha_1,\alpha_2)$. (By pseudo-coset we
  mean that $\ast$ is either a single double coset or is empty.)  This
  pseudo-coset corresponds precisely to the constant objects, and is
  therefore functorial in $F_{\ms Z}$.

  With this notation, the criterion of the lemma states that for any
  pair of objects $\beta_i\in\ms X_{T_i}$, $i=1,2$, with images
  $\alpha_i\in\ms Y_{T_i}$, the map
  \begin{equation}\label{eq:blobby}
\aut(\beta_2)\backslash\isom(\beta_1|_U,\beta_2|_U)/\aut(\beta_1)\to\aut(\alpha_2)\backslash\isom(\alpha_1|_U,\alpha_2|_U)/\aut(\alpha_1)
\end{equation}
has the propery that the full preimage of $\ast$ is $\ast$.  In
particular, if an isomorphism $\phi:\beta_1|_U\simto\beta_2|_U$ has
image $f(\phi)$ which extends to an isomorphism
$\psi:\alpha_1\to\alpha_2$, we see that $\phi$ must extend to an
isomorphism $\widebar{\phi}:\beta_1\to\beta_2$.  It then follows from
the separatedness of the diagonals of $\ms X$ and $\ms Y$ that
$\widebar{\phi}$ maps to $\psi$ under $f$.  This is precisely the
valuative criterion given in Proposition 7.8 of \cite{l-m-b}.
 
Conversely, suppose $f$ is separated.  The valuative criterion
[\emph{loc.\ cit.\/}] can be stated as follows: given a trait $Q$ with
generic point $U$ and two objects $\beta_1$ and $\beta_2$ of $\ms X_Q$
with images $\alpha_1$ and $\alpha_2$ in $\ms Y_Q$, any isomorphism
$\phi:\beta_1|_U\simto\beta_2|_U$ whose image in $\ms Y_U$ extends to
an isomorphism $\alpha_1\to\alpha_2$ must extend to an isomorphism
$\beta_1\simto\beta_2$.  But this property only depends upon the image
of $\phi$ (resp.\ $f(\phi)$) in the double coset space
$\aut(\beta_2)\backslash\isom(\beta_1|_U,\beta_2|_U)/\aut(\beta_1)$
(resp.\
$\aut(\alpha_2)\backslash\isom(\alpha_1|_U,\alpha_2|_U)/\aut(\alpha_1)$).
Thus, we find that the preimage of $\ast$ under the natural map
(\ref{eq:blobby}) is $\ast$, as desired.
\end{proof}

We next move on to properness. 

\begin{lem} Suppose $f:\ms X \to \ms Y$ is a morphism of Artin stacks.
Given a trait $T$ with generic point $U$, and a $2$-commutative diagram
$$\xymatrix{{U} \ar[r] \ar[d] & {\ms X} \ar[d]^{f} \\
{T} \ar[r]\ar@{-->}[ur] & {\ms Y},}$$
consider the induced commutative diagram 
$$\xymatrix{{U} \ar[r] \ar[d] & {F_{\ms X}} \ar[d]^{F_f} \\
{T} \ar[r]\ar@{-->}[ur] & {F_{\ms Y}}.}$$
Then:
\begin{enumerate}
\item every square of the second form is induced by one of the first form;
\item if the first square admits a morphism $T \to \ms X$ making the
entire diagram $2$-commutative, then the induced map $T \to F_{\ms X}$
gives a commutative diagram when added to the second square;
\item if $\ms Y$ has proper inertia, then
  we have conversely that any morphism $T \to F_{\ms X}$ making the
  second diagram commutative yields a morphism $T \to \ms X$ making
  the first diagram $2$-commutative.
\end{enumerate}
\end{lem}

\begin{proof} The first assertion is trivial. For the remaining claims,
the key issue to consider is that, as discussed in Remark 
\ref{rem:2-diag-objs}, the $2$-commutative square
above consists of $\mu_0 \in \ms X_U$, $\eta \in \ms Y_T$, and an
isomorphism $\alpha:f(\mu_0) \to \eta|_U$, and a map $T \to \ms X$,
given by $\mu \in \ms X_T$, allows the diagram to be filled in to
a $2$-commutative diagram if there exist isomorphisms 
$\beta:\mu_0 \to \mu|_U$ and $\gamma: f(\mu) \to \eta$ such that
$\alpha=\gamma|_U \circ f(\beta)$. Filling in the second diagram
is the same, except without the final compatibility condition on 
the isomorphisms. Thus, it is clear that if the first diagram may
be filled in to be $2$-commutative, the second one may be filled in
to be commutative. Conversely, if the second one may be filled in
to be commutative, 
$\gamma|_U \circ f(\beta) \circ \alpha^{-1} \in \Aut(\eta|_U)$, and
if $\Aut(\eta) \to \Aut(\eta|_U)$ is surjective, we can modify $\gamma$ to 
obtain
$\alpha=\gamma|_U \circ f(\beta)$, giving the desired $2$-commutativity. 
\end{proof}

An almost immediate consequence of the two lemmas is the following.

\begin{cor}\label{C:proper-easy} 
  Properness is isonatural for morphisms $f:\ms X \to \ms Y$ of Artin
  stacks, where $\ms Y$ is further supposed to be locally Noetherian
  and to have proper inertia.

  Specifically, $f$ is proper if and only if it is locally of finite
  type, quasi-compact, and separated, and if for every trait $T$ with
  generic point $U$, with morphisms $T \to F_{\ms Y}$ and $U \to
  F_{\ms X}$, there exists a trait $T'$ with generic point $U'$,
  obtained by normalizing $T'$ inside the finite field extension given
  by $U' \to U$, and morphisms making the following diagram commute:
$$\xymatrix{{U'} \ar[r] \ar[d] & {U} \ar[r] \ar[d] & 
  {F_{\ms X}} \ar[d]^{F_f} \\
  {T'} \ar[r]\ar[ur] & {T} \ar[r] & {F_{\ms Y}}.}$$
\end{cor}

\begin{proof} 
  We first remark that there is a standard stack version of the
  valuative criterion for properness. This is stated as (iii) of
  Theorem 7.10 of \cite{l-m-b}, using also Proposition 7.12 of {\it
    loc.\ cit.}, and noting that condition (*) of {\it ibid.}\ is
  always satisfied, thanks to the main theorem of \cite{ol4}.

  Because being separated, quasi-compact, or locally of finite type
  are all isonatural, we need only check that our asserted valuative
  criterion is equivalent to the usual valuative criterion cited
  above.  But that follows immediately from the previous lemma.
\end{proof}

Note that in particular, if $\ms Y$ is locally Noetherian scheme or
algebraic space, properness is isonatural. We also use doubled traits
to see that having proper inertia is isonatural, which completes our list
of isonatural properties of stacks.

\begin{proof}[Proof of Corollary \ref{cor:summary-absolute}] That being
an Artin or Deligne-Mumford stack is isonatural is part of Corollary
\ref{C:recog-covers}. Next, because a stack is a gerbe (over an
algebraic space) if and only if the sheafification of the associated
functor is an algebraic space (see Remark 3.16(1) \cite{l-m-b}), we
also see that the property of being a gerbe is isonatural. 

Corollary \ref{cor:smooth-local-abs} implies that locally Noetherian,
normal, reduced, regular, and quasi-compact are each isonatural for Artin
stacks.

Finally, we see that having proper inertia is likewise isonatural,
because if $\ms X$ be a quasi-algebraic stack, $T_1$ a trait with
generic point $U$, and $\eta \in \ms X_{T_1}$, if we let $T$ be the
doubled trait obtained by gluing $T_1$ to itself along $U$, it is easy to
see that $\Aut(\eta) \to \Aut(\eta|_U)$ is surjective if and only if there 
is a unique element of $F_{\ms X}(T)$ pulling back to $\eta$ under
both restriction maps.
\end{proof}

By using doubled traits in the criterion for universal closedness, we can 
further expand the range of cases in which we can treat properness, as 
follows.

\begin{prop}\label{P:proper-abelian} Properness is isonatural for 
morphisms $f: \ms X \to \ms Y$ of Artin stacks, with $\ms Y$ Noetherian 
and abelian (see Definition \ref{d:notat-basic-defin-1}).

Specifically, $f$ is proper if and only if it satisfies all the conditions
of Corollary \ref{C:proper-easy}, and if in addition, 
for every doubled trait $T$ with generic point $U$, obtained by gluing
together traits $T_1=T_2$ along $U$, and given morphisms 
$T \to F_{\ms Y}$ and $T_1 \to F_{\ms X}$, there exists a doubled trait 
$T'$ with generic point
$U'$, obtained from $T'_1=T'_2$ the normalization of $T_1=T_2$ inside the 
finite field extension given by $U' \to U$,
and morphisms making the following diagram commute:
$$\xymatrix{{T_1'} \ar[r] \ar[d] & {T_1} \ar[r] \ar[d] & 
{F_{\ms X}} \ar[d]^{F_f} \\
{T'} \ar[r]\ar[ur] & {T} \ar[r] & {F_{\ms Y}}.}$$
\end{prop}

\begin{proof} As before, it suffices to see that our criterion in terms
of functors and doubled traits is equivalent to the usual criterion in
terms of stacks, under the hypothesis that $\ms Y$ has abelian 
stabilizers. Once $T'$ is given, we 
can ignore the original square, and consider instead the square
$$\xymatrix{{T_1'} \ar[r] \ar[d] & {F_{\ms X}} \ar[d]^{F_f} \\
{T'} \ar[r]\ar@{-->}[ur] & {F_{\ms Y}}.}$$
To simplify notation and avoid the uncontrolled proliferation of
$'$, when we say ``after extension'' $U' \to U$ we
will assume we have replaced $U$ by $U'$, $T$ by $T'$, objects and
morphisms by their appropriate pullbacks, and so forth. We will use
Remark \ref{rem:2-diag-objs} to translate between the $2$-commutative
diagrams of the formal criterion and objects and isomorphisms of the
stacks themselves.

The map $T_1 \to F_{\ms X}$ is equivalent to an object 
$\mu_1 \in \ms X_{U}$, up to isomorphism.
The map $T \to F_{\ms Y}$ is equivalent to a pair of objects
$\eta_1 \in \ms Y_{T_1}, \eta_2 \in \ms Y_{T_2}$, and a choice 
of isomorphism $\varphi:\eta_1 |_{U} \risom \eta_2 |_{U}$, up to 
simultaneous isomorphism commuting with $\varphi$. We also assume
that $f(\mu_1) \cong \eta_1$. The desired map 
$T \to F_{\ms X}$ is then given by extending $\mu_1$ to a triple 
$(\mu_1,\mu_2,\phi)$ for $\mu \in \ms X_{T_2}$, 
$\phi:\mu_1|_U \risom \mu_2|_U$, 
with the additional restriction that there exist isomorphisms
$\alpha_1:f(\mu_1) \to \eta_1$ and $\alpha_2:f(\mu_2) \to \eta_2$,
satisfying $\varphi \circ \alpha_1|_{U} = \alpha_2|_{U} \circ f(\phi)$. 

Suppose $f$ is proper. We have by the earlier lemma that the conditions
of Corollary \ref{C:proper-easy} are satisfied, so we wish to show that
our condition on doubled traits is also satisfied. Starting with
$(\eta_1,\eta_2,\varphi)$ in $\ms Y_T$ and $\mu_1 \in \ms X_{T_1}$, and
fixing further any $\alpha_1:f(\mu_1) \to \eta_1$, applying the valuative 
criterion of
properness to $\eta_2$, $\mu_1|_U$, and $\varphi \circ \alpha_1|_U$,
after an appropriate extension there exists 
$\mu_2 \in \ms X_{T_2}$, $\beta_2:\mu_1|_U \to \mu_2|_U$, 
$\gamma_2:f(\mu_2) \to \eta_2$, such that 
$\varphi \circ \alpha_1|_{U}=\gamma_2|_{U} \circ f(\beta_2)$.
Setting the above $\alpha_2=\gamma_2$ and $\phi=\beta_2$ gives us
precisely what we wanted.
Note that this direction did not use any hypotheses on the stabilizer
being abelian.

Conversely, suppose that $f$ satisfies our criterion, and $\ms Y$ has
abelian stabilizer groups. We then want to show that $f$ satisfies
the valuative criterion for universal closedness, and is therefore
proper. Here, we are simply given $\mu_0 \in \ms X_{U}$, 
$\eta \in \ms Y_{T_1}$, and $\beta: f(\mu_0) \to \eta|_{U}$, and we
wish to show that after finite extension, there exists
$\mu \in \ms X_{T_1}$ and isomorphisms $\gamma:\mu_0 \to \mu|_{U}$
and $\alpha:f(\mu) \to \eta$ such that 
$\alpha|_{U} \circ f(\gamma) = \beta$. We first apply the criterion
of Corollary \ref{C:proper-easy} to find that after extension, we have
$\mu_1 \in \ms X_{T_1}$ and isomorphisms $\gamma_1:\mu_0 \to \mu_1|_{U}$ 
and $\alpha':f(\mu_1) \to \eta$ not necessarily satisfying any compatibility
condition. We set $\eta_1=\eta_2=\eta$, and 
$\varphi:\beta \circ f(\gamma_1)^{-1} \circ (\alpha'|_{U})^{-1}$. Our 
criterion says that after an additional extension, we have
$\mu_2 \in \ms X_{T_2}$, $\phi: \mu_1|_{U} \to \mu_2|_{U}$, and
$\alpha_i: f(\mu_i) \to \eta_i$ satisfying $\varphi \circ \alpha_1|_{U}
=\alpha_2|_{U} \circ f(\phi)$, so we have 
$\beta \circ f(\gamma_1)^{-1} \circ (\alpha'|_{U})^{-1} 
\circ \alpha_1|_{U}= \alpha_2|_{U} \circ f(\phi)$. We claim 
that if we set $\mu=\mu_2$, and $\gamma= \phi \circ \gamma_1$, and
$\alpha=\alpha' \circ \alpha_1^{-1} \circ \alpha_2$, we obtain
$\alpha|_{U} \circ f(\gamma) = \beta$, as desired.
The key observation is that
$$\alpha=\alpha_2 \circ (\alpha_2^{-1} \circ \alpha_1) \circ 
(\alpha_1^{-1} \circ \alpha') \circ (\alpha_1^{-1} \circ \alpha_2).$$
Because automorphism groups are abelian, conjugation of an 
automorphism by any two isomorphisms yields the same result, and applying
this to $(\alpha_1^{-1} \circ \alpha') |_U$ we find
$$\alpha|_{U}=\alpha_2|_{U} \circ f(\phi) \circ 
(\alpha_1^{-1} \circ \alpha') |_U \circ f(\phi)^{-1}.$$
Substituting above we easily obtain the desired identity.
\end{proof}

We have now finished proving isonaturality of all the asserted properties
of Artin stacks.

\begin{proof}[Proof of Theorem \ref{thm:summary-mors}] Isonaturality for
morphisms being locally of finite presentation is Proposition 
\ref{P:recog-lfp}, smoothness follows from Corollary 
\ref{C:recog-covers}, 
and then locally of finite type, surjective, and flat follow from Corollary 
\ref{C:kitchen-sink}. Isonaturality for morphisms being quasi-compact
is Proposition \ref{P:qc-funcy}, and separated and locally of finite type 
when the target is locally Noetherian is Lemma \ref{L:separated-funcy}. 
Finally, proper morphisms when the target is locally Noetherian and
has proper inertia is covered by Corollary \ref{C:proper-easy}, and when the 
target is locally Noetherian and has abelian stabilizers is Proposition 
\ref{P:proper-abelian}.
\end{proof}

\subsection{Functorial reconstruction of automorphism groups}
\label{sec:nanas}

In this section we describe a structure that can be used to recover
the presheaf of conjugacy classes in the inertia of any
quasi-algebraic stack.  When the stack is abelian, this permits us to
reconstruct abelian automorphism sheaves.

\begin{defn} The {\it universal binana\/} $N_{2,\ZZ}$ is the proper curve 
over $\Spec \ZZ$ obtained by gluing together two copies of $\PP^1_{\ZZ}$
to one another transversally along the $0$ and $1$ sections. Given any 
scheme $T$, the {\it binana over $T$\/}, denoted $N_{2,T}$ is 
$N_{2,\ZZ} \times T$. We denote by $0_T$ and $1_T$ the images of the 
$0$ and $1$ sections.

The binana over $T$ has the two {\it peel\/} maps $P^2_i:\PP^1_T \to N_{2,T}$ 
for $i=1,2$;
each is a closed immersion, and the intersection of their images is
precisely $0_T \cup 1_T$. 

Finally, we set the following notation: 
$N_{2,T}^0:=N_{2,T}\smallsetminus 0_T$,
$N_{2,T}^1:=N_{2,T}\smallsetminus 1_T$,
and
$N_{2,T}^{0,1}:=N_{2,T}\smallsetminus \{0_T,1_T\}$.
\end{defn}

We consider objects of functors over binanas which are constant on each
peel; the isomorphism classes can thus be thought of (at least informally)
in terms of gluing along isomorphisms over the $0$ and $1$ sections.

\begin{defn} Let $F$ be a functor from $S$-schemes to sets, $T$ a scheme
over $S$, and $\eta \in F(T)$. Given a $T$-scheme $T'$, we say an object 
$\eta' \in F(T')$ is {\it $\eta$-trivial\/} if $\eta'=\eta|_{T'}$. 
\end{defn}

\begin{defn}
Given $\eta\in F(T)$, an {\it $\eta$-binana\/} is an object
$\tilde{\eta} \in F(N_{2,T})$ satisfying the following conditions:
\begin{enumerate}
\item $\tilde{\eta}|_{N_{2,T}\smallsetminus 0_T}$ and 
$\tilde{\eta}|_{N_{2,T}\smallsetminus 1_T}$ are both $\eta$-trivial;
\item $(P^2_{1})^\ast(\tilde{\eta})$ and
  $(P^2_{2})^\ast(\tilde\eta)$ are $\eta$-trivial.
\end{enumerate}
\end{defn}

\begin{defn}
  The functor sending $T'\to T\in T\Sch$ to the set of
  $\eta_{T'}$-binanas will be called the \emph{functor of $\eta$-binanas\/} 
  and denoted $\Bin(\eta)$.
\end{defn}

\begin{defn}\label{sec:binana-recovers-band}
  Let $G$ be a sheaf of groups on a site $\Xi$.  The presheaf sending
  $R$ in $\Xi$ to the set of conjugacy classes of $G(R)$ will be
  called the \emph{presheaf of conjugacy classes of $G$\/} and 
  denoted $\Conj(G)$.
\end{defn}

\begin{lem}
  If the sheaf of groups $G$ in Definition
  \ref{sec:binana-recovers-band} is abelian then there is a canonical
  isomorphism of presheaves $G\to\Conj(G)$.
\end{lem}
\begin{proof}
  This follows immediately from the definition.
\end{proof}

\begin{prop}\label{P:binana-bipoints-1} 
Let $\ms X$ be a quasi-algebraic stack and $\widetilde a\in \ms X_T$ an
  object with isomorphism class $a\in F_{\ms X}(T)$.
  \begin{enumerate}
  \item there is a canonical isomorphism of functors 
$$\Bin(a)\simto\Conj(\aut(\widetilde a));$$
\item if $\aut(\widetilde a)$ is an abelian sheaf there is a canonical
  isomorphism $$\Bin(a)\simto\aut(\widetilde a).$$
  \end{enumerate}
Moreover, these isomorphisms are functorial in the pair $(\ms
X,\widetilde a)$.
\end{prop}

\begin{proof} The hypothesis that $\ms X$ is quasi-algebraic implies
in particular that $\Aut(\widetilde a)$ is a group scheme over $T$.

Condition (1) in the definition of an $a$-binana implies
that $a$-binanas may be understood in terms of gluing $a$-trivial
families on $N_{2,T}^0$ and $N_{2,T}^1$ along the intersection
$N_{2,T}^{0,1}$, which is isomorphic to 
$(\PP^1_T \smallsetminus \{0_T, 1_T\})\coprod  
(\PP^1_T \smallsetminus \{0_T, 1_T\}).$
Thus, an $a$-binana is determined by the data of two sections 
$\varphi_1,\varphi_2$ of 
$\Aut(\widetilde a)$ over $\PP^1_T \smallsetminus \{0_T, 1_T\}$; we think of the
pair $(\varphi_1,\varphi_2)$ as a section of $\Aut(\widetilde a)$ over
$N_{2,T}^{0,1}$. Condition (2) 
is precisely the restriction that each of these sections must be expressible 
as the difference of sections of $\Aut(\widetilde a)$ over 
$\PP^1_T\smallsetminus 0_T$ and over $\PP^1_T\smallsetminus 1_T$. 

Two $a$-binanas obtained from gluing along $(\varphi_1,\varphi_2)$
and $(\varphi'_1,\varphi'_2)$ are isomorphic if and only if there exists
sections $\alpha_0,\alpha_1$ of $\Aut(\widetilde a)$ over 
$N_{2,T}^0$ and $N_{2,T}^1$ respectively, such that 
$(\varphi'_1,\varphi'_2) \circ \alpha_0|_{N_{2,T}^{0,1}} =
\alpha_1|_{N_{2,T}^{0,1}} \circ (\varphi_1,\varphi_2)$.

We now construct a map from $\Aut(\widetilde a)$ to the set of $a$-binanas.
Given $\varphi \in \Aut(\widetilde a)$ (over the base scheme $T$ itself), 
we glue along the constant automorphisms $(\id, \varphi)$ to obtain
a binana. Being constant, there is no problem with extending either of
them to $\PP^1_T$, so condition (2) is satisfied, and we obtain an 
$a$-binana. We wish to show that two binanas obtained in this way
from $\varphi$ and $\varphi'$ are the same if and only if $\varphi$
and $\varphi'$ are conjugate to one another in $\Aut(\widetilde a)$, and that
every $a$-binana is obtained in this way.

For the first assertion, $\varphi$ and $\varphi'$ yield the same 
$a$-binana if and only if there exist $\alpha_0$ and $\alpha_1$ as
above with 
$(\id,\varphi) \circ \alpha_0|_{N_{2,T}^{0,1}} =
\alpha_1|_{N_{2,T}^{0,1}} \circ (\id,\varphi')$, which is equivalent to
$\alpha_0=\alpha_1$ after restriction to the first copy of 
$\PP^1_T \smallsetminus \{0_T,1_T\}$, and 
$\alpha_0=\varphi^{-1} \alpha_1 \varphi'$ after restriction to the second
copy of
$\PP^1_T \smallsetminus \{0_T,1_T\}$. 

Suppose such $\alpha_i$ exist. Since $\varphi$ and $\varphi'$ are
constant, this implies that $\alpha_0$ and $\alpha_1$ can be extended
over the partial normalizations of $N_{2,T}$ over $0_T$ and $1_T$
respectively, which implies they are themselves constant, since
$\Aut(\widetilde a)$ is a group scheme.  Hence, by looking at the
first copy of $\PP^1_T \smallsetminus \{0_T,1_T\}$, the $\alpha_i$ are
also globally equal. Looking at the second copy of $\PP^1_T
\smallsetminus \{0_T,1_T\}$ gives us $\varphi=\alpha_1 \varphi'
\alpha_0^{-1}=\alpha_0 \varphi' \alpha_0^{-1}$, and since $\alpha_0$
is constant, we find that $\varphi$ and $\varphi'$ are conjugate, as
desired. Conversely, it is clear that if $\varphi = \alpha \varphi'
\alpha^{-1}$, setting $\alpha_0$ and $\alpha_1$ equal to the constant
sections obtained from $\alpha$ yields an isomorphism between the
$a$-binanas obtained from $\varphi$ and $\varphi'$.

It remains to see that given an $a$-binana coming from a pair 
$(\varphi_1,\varphi_2)$, there is some $\varphi \in \Aut(\widetilde a)$ yielding
the same $a$-binana. By hypothesis, there exist 
$\alpha_{0,1},\alpha_{0,2}$ sections of $\Aut(\widetilde a)$ over 
$\PP^1_T\smallsetminus 0_T$ and 
$\alpha_{1,1},\alpha_{1,2}$ sections of $\Aut(\widetilde a)$ over
$\PP^1_T\smallsetminus 1_T$ with 
$\varphi_1=\alpha_{1,1}^{-1} \circ \alpha_{0,1}$ and 
$\varphi_2=\alpha_{1,2}^{-1} \circ \alpha_{0,2}$ after restriction to 
$\PP^1_T \smallsetminus \{0_T,1_T\}$. If we modify $\alpha_{0,2}$
and $\alpha_{1,2}$ by the constant sections coming from 
$\alpha_{0,1}|_{0_T} \circ \alpha_{0,2}^{-1}|_{0_T}$, we can glue
$\alpha_{0,1}$ and $\alpha_{0,2}$ to obtain a section $\alpha_0$ of
$\Aut(\widetilde a)$ over $N_{2,T}^0$. Define $\alpha_1$ over
$N_{2,T}^1$ to be obtained by gluing $\alpha_{1,1}$ to 
$(\alpha_{1,1}|_{1_T} \circ \alpha_{1,2}|_{1_T}^{-1})
\circ \alpha_{1,2}$.
Then we see that $\alpha_0$ and $\alpha_1$ define an isomorphism between
the $a$-binana obtained by gluing along $(\varphi_1,\varphi_2)$
and the one obtained from 
$\alpha_{1,1}|_{1_T} \circ \alpha_{1,2}|_{1_T}^{-1} \in \Aut(\widetilde a)$.
\end{proof}

In order to recover the composition law on automorphism groups,
we now introduce a further structure.

\begin{defn} The {\it universal trinana\/} $N_{3,\ZZ}$ is the proper curve 
over $\Spec \ZZ$ obtained by gluing together three copies of 
$\PP^1_{\ZZ}$ transversally along the $0$ and $1$ sections. Given a
scheme $T$, the {\it trinana over $T$\/}, denoted $N_{3,T}$, is
$N_{3,\ZZ} \times T$. As before, we denote by $0_T$ and $1_T$ the images 
of the $0$ and $1$ sections.

The trinana over $T$ has three peel maps $P^3_i:\PP^1_T \to N_{3,T}$ for
$i=1,2,3$; 
each is again a closed immersion, and the intersection of any two of 
their images is precisely $0_T \cup 1_T$.
Finally, there are three {\it bipeel\/} maps
$P_{i,j}:N_{2,T} \to N_{3,T}$ for $(i,j)=(1,2),(1,3),(2,3)$. Each is
again a closed immersion, and we have $P_{i,j} \circ P^2_1=P^3_i$
and $P_{i,j} \circ P^2_2=P^3_j$.

Finally, we set the following notation: 
$N_{3,T}^0:=N_{3,T}\smallsetminus 0_T$,
$N_{3,T}^1:=N_{3,T}\smallsetminus 1_T$,
and
$N_{3,T}^{0,1}:=N_{3,T}\smallsetminus \{0_T,1_T\}$.
\end{defn}

\begin{defn}
  Given $\eta\in F(T)$, an \emph{$\eta$-trinana\/} is an object $\eta'$ of
  $F(N_{3,T})$ such that 
  \begin{enumerate}
  \item $\tilde{\eta}|_{N_{3,T}\smallsetminus 0_T}$ and 
$\tilde{\eta}|_{N_{3,T}\smallsetminus 1_T}$ are both $\eta$-trivial;
\item $(P^3_i)^{\ast}(\eta')$ is $\eta$-trivial for $i=1,2,3$.
  \end{enumerate}
\end{defn}

\begin{defn}
  Given $\eta\in F(T)$, the functor which assigns to any $T'\to T\in
  T\Sch$ the set of $\eta_{T'}$-trinanas will be called the
  \emph{functor of $\eta$-trinanas\/} and denoted $\Tri(\eta)$.
\end{defn}

\begin{defn} Given $\eta\in F(T)$ as above, an $\eta$-binana
  $\tilde{\eta}$ and an $\eta$-binana 
$\tilde{\eta}'$, an {\it  
$(\tilde{\eta},\tilde{\eta}')$-trinana\/} is an object 
$\mu \in F(N_{3,T})$ such that:
\begin{enumerate}
\item $\mu|_{N_{3,T}^0}$ and $\mu|_{N_{3,T}^1}$
are both $\eta$-trivial; 
\item we have
$P_{1,2}^* \mu = \tilde{\eta}$ and $P_{2,3}^* \mu = \tilde{\eta}'$.
\end{enumerate}
\end{defn}

\begin{defn}\label{sec:definitions-1}
  The three bipeel morphisms yield a diagram of functors
\begin{equation}
  \label{eq:2}
  \xymatrix{\Tri(\eta)\ar[rr]^-{P_{1,2}^\ast\times P_{2,3}^\ast}\ar[d]_{P_{1,3}^{\ast}}&&\Bin(\eta)\times\Bin(\eta)\\
    \Bin(\eta)&& }
\end{equation}
which we will call the \emph{fundamental diagram of $\eta$-nanas\/}.
\end{defn}

\begin{prop}\label{P:gp-law-nana}
  Given a quasi-algebraic stack $\ms X$, a scheme $T$, and $\widetilde 
  a \in \ms X_T$ with image $a\in F_{\ms X}(T)$, the group
  $\Aut(\widetilde a)$ is abelian if and only if the horizontal arrow
  in the fundamental diagram of Definition \ref{sec:definitions-1} is
  a bijection.  In this case, the composition law is given by the
  vertical arrow in the fundamental diagram, via the
  isomorphism of Proposition \ref{P:binana-bipoints-1}(2).
\end{prop}

\begin{proof} As in the case of $a$-binanas, we see that condition
(1) for a trinana means that it is determined by a triple of sections 
$(\psi_1, \psi_2, \psi_3)$ of $\Aut(\widetilde{a})$ over 
$\PP^1_T\smallsetminus \{0_T,1_T\}$. If $b$ and $b'$
are $a$-binanas represented by $(\varphi_1,\varphi_2)$ and 
$(\varphi_1',\varphi_2')$
respectively, then condition (2) simply requires isomorphisms between
the binanas obtained from $(\psi_1,\psi_2)$ and $(\varphi_1,\varphi_2)$,
and $(\psi_2,\psi_3)$ and $(\varphi_1',\varphi_2')$. Moreover, we know
from the proof of the previous proposition that without loss of generality,
we can set $(\varphi_1,\varphi_2)=(\id,\varphi)$ and 
$(\varphi_1',\varphi_2')=(\id,\varphi')$, where $\varphi$ and $\varphi'$
are constant sections of $\Aut(\widetilde{a})$. Now, it is easy to check that
if we set $(\psi_1,\psi_2,\psi_3)=(\id,\varphi_1,\varphi_1 \varphi_2)$
we obtain an $(b,b')$-trinana, so our assertion is 
that this is the unique possibility if and only if $\Aut(\widetilde{a})$ is abelian. 

One direction is clear: if $\Aut(\widetilde{a})$ is 
non-abelian, then by choosing $\varphi$ in a non-trivial conjugacy class, 
say with $\gamma^{-1} \varphi \gamma \neq \varphi$, then we see by
comparing the two representations of the same $a$-binanas given by
$\varphi'=\varphi^{-1}$ and $\varphi'=\gamma^{-1} \varphi^{-1} \gamma$,
that we have the two $(b,b')$-trinanas given by
$(\id, \varphi, \id)$ and 
$(\id, \varphi, \varphi \gamma^{-1} \varphi^{-1} \gamma)$, and these
cannot be isomorphic because their pullbacks under $P_{1,3}$ yield
non-isomorphic $a$-binanas. 

It remains to show that if $\Aut(\widetilde{a})$ is abelian, then an
$(b,b')$-trinana given by 
$(\psi_1,\psi_2,\psi_3)$ is necessarily isomorphic to the one given
by $(\id,\varphi_1,\varphi_1 \varphi_2)$. We therefore wish to 
construct $\beta_0$ and $\beta_1$, sections of $\Aut(\widetilde{a})$ over
$N_{3,T}^0$ and $N_{3,T}^1$ respectively, such that
$(\psi_1,\psi_2,\psi_3) \circ \beta_0|_{N_{3,T}^{0,1}} =
\beta_1|_{N_{3,T}^{0,1}} \circ (\id,\varphi_1,\varphi_1 \varphi_2)$.
We are given $\alpha_0$ and $\alpha_0'$ over $N_{2,T}^0$ and
$\alpha_1$ and $\alpha_1'$ over $N_{2,T}^1$, such that 
$(\psi_1,\psi_2) \circ \alpha_0|_{N_{2,T}^{0,1}} =
\alpha_1|_{N_{2,T}^{0,1}} \circ (\id, \varphi_1)$ and
$(\psi_2,\psi_3) \circ \alpha_0'|_{N_{2,T}^{0,1}} =
\alpha_1'|_{N_{2,T}^{0,1}} \circ (\id,\varphi_2)$.

We define $\beta_0$ to be $\alpha_0$ on the first and second peels,
and $\alpha_0'(\alpha_0'|_{0_T})^{-1} \alpha_0|_{0_T}$ on the
third peel. Similarly, we set $\beta_1$ to be $\alpha_1$ on the first 
and second peels, and $\alpha_1'(\alpha_1'|_{1_T})^{-1} \alpha_1|_{1_T}$ 
on the third peel. Using the abelian hypothesis, we note that
on the second peel, restricted to $\PP^1_T\smallsetminus\{0_T,1_T\}$,
we have $\alpha_1'\alpha_1^{-1} =\varphi_1 \alpha_0' \alpha_0^{-1}$.
Since $\varphi_1$ is constant, this means that $\alpha_1' \alpha_1^{-1}$
on the second peel extends over $0_T$ to all of $\PP^1_T$, and must
therefore be constant. Thus, $\alpha_0'\alpha_0^{-1}$ is also constant,
and we conclude the identity 
$(\alpha_1|_{1_T})^{-1} \alpha_1'|_{1_T} 
(\alpha_0'|_{0_T})^{-1} \alpha_0|_{0_T} = \varphi_1.$ From this, it
is easy to check that $\beta_0$ and $\beta_1$ define the required
isomorphism of trinanas.
\end{proof}

\begin{remark} It is a general fact that any morphism of stacks which
induces a bijection on isomorphism classes and isomorphisms on all
automorphism groups is an isomorphism. It thus follows from the
previous propositions that if we have a morphism $f:\ms X \to \ms Y$
of quasi-algebraic stacks inducing an isomorphism 
$F_{\ms X} \risom F_{\ms Y}$, and if either $\ms X$ or $\ms Y$ is
abelian, then $f$ is an isomorphism.
\end{remark}

\begin{cor}\label{cor:band-love}
  If $\ms X\to X$ is an abelian quasi-algebraic gerbe then the band of
  $\ms X$ can be recovered from $F_{\ms X}$.
\end{cor}
\begin{proof}
Write $\Phi$ for the category fibered in groupoids on $X$ associated to 
the functor $F_{\ms X}$, so that there is a diagram of functors
$$\xymatrix{\ms X\ar[r]^c & \Phi\ar[r]^p & X\Sch.}$$  Let $\Ab$ be the
category of abelian groups.
 
Viewing the inertia stack of $\ms X$ as a sheaf on the natural site of
$\ms X$ yields a functor $\iota:\ms X^{\circ}\to\Ab$.  According to
Propositions \ref{P:binana-bipoints-1} and \ref{P:gp-law-nana}, there is 
a functor $\Gamma:\Phi^{\circ}\to\Ab$
such that $\iota$ is isomorphic to $\Gamma\circ c$.  (The underlying
set of $\Gamma$ is just $\Bin$.)
 
Moreover, since $\ms X$ is an abelian gerbe, there is an abelian sheaf
$\Lambda:X\Sch^{\circ}\to\Ab$ on $X$ (the band of $\ms X$) and an 
isomorphism $\iota\simto
\Lambda\circ p\circ c$.  We find an isomorphism $\psi:\Gamma\circ
c\simto\Lambda\circ p\circ c$.  Let $\chi:U\to X$ be an fppf covering such
that there is a lift $q:U\to\ms X$; let $\widebar q$ denote the
composition $c\circ q$. Composing with $\psi$ yields an
isomorphism $\Gamma\circ c\circ q\simto\Lambda\circ p\circ c\circ q$,
which (via the natural isomorphisms) yields an isomorphism
$\Gamma\circ \widebar q\simto\Lambda\circ \chi$.
 
This isomorphism tells us that (via $\Bin$ and diagram \eqref{eq:2})
we can recover $\Lambda\circ\chi$ for some fppf covering $\chi:U\to
X$.  Since any abelian sheaf on $X$ is uniquely determined by its
values on the category of $U$-schemes (a simple consequence of the
sheaf property), it follows that $\Lambda$ is uniquely determined up
to isomorphism by $F_{\ms X}$, as desired.
\end{proof}
 
\begin{cor}\label{C:sheafy}
The associated functor of a quasi-algebraic stack $\ms X$ is a sheaf if 
and only if $\ms X$ has no non-trivial automorphisms.

In particular, if $\ms X$ is an Artin stack, then the associated functor
is a sheaf if and only if $\ms X$ is an algebraic space.
\end{cor}

\begin{proof} Certainly, if $\ms X$ has no non-trivial automorphisms, then
the stack condition implies its associated functor is a sheaf. Conversely,
if $\ms X$ is a sheaf, we see from the above argument that every 
automorphism group must be trivial: if $\Aut(\eta) \neq \{1\}$ for
some $T$ and $\eta\in\ms X_T$, we would have at least two non-isomorphic 
$\eta$-binanas, which by definition become isomorphic after restriction
to $N_{2,T}^0$ and $N_{2,T}^1$. This would violate the sheaf condition, 
so $\Aut(\eta)=\{1\}$. 

For the last assertion, we use that by Corollary 8.1.1 of \cite{l-m-b}, 
an Artin stack is an algebraic space if and only if it has no non-trivial
automorphisms. 
\end{proof}

\section{Isonatural stacks}\label{sec:isonatural-stacks}

In this section, we examine several classes of stacks, showing that
within these classes, stacks are uniquely determined by their associated
functors. We also show that one can recognize whether a given 
quasi-algebraic stack lies in each class, proving that a stack lying
in any of the given classes is isonatural.

\subsection{Summary of results}\label{sec:stacks-summary}

In order to give the precise statements of our results, we make the
following definitions.

\begin{defn} We say that an algebraic space $X$ is \emph{strongly R1}
if $X$ is Noetherian, integral, separated, and R1. 
\end{defn}

\begin{defn}\label{def:brauer}
  Let $G=\bigoplus\m_N$ be a diagonalizable group scheme.  A
  cohomology class $\alpha\in\H^2(X,G)=\bigoplus\H^2(X,\m_N)$ will be
  called {\it Brauer\/} if the image of each
  component via $\H^2(X,\m_N)\to\H^2(X,\G_m)$ lies in $\Br(X)$.  A
  $G$-gerbe $\ms X$ will be called Brauer if its cohomology class
  $[\ms X]\in\H^2(X,G)$ is Brauer.
\end{defn}

\begin{remark}
  According to a theorem of Gabber \cite{dejong-gabber}, if the
  connected components of $X$ are 
  quasi-compact separated schemes admitting ample invertible sheaves 
  then every class as in Definition \ref{def:brauer} is Brauer.

  Furthermore, whether or not a cohomology class is Brauer is 
  independent of the choice of representation of $G$ as a direct sum.
\end{remark}

\begin{defn}\label{D:cl}
  Given a quasi-algebraic stack $\ms X$, the \emph{clean locus} of $\ms X$,
  denoted $\clean(\ms X)$ is the locus over which the inertia stack 
  $\ms I(\ms X)\to\ms X$ is an isomorphism (i.e., the locus parametrizing 
  objects with trivial automorphism sheaves) 
\end{defn}

\begin{prop} If $\ms X$ has proper inertia then the clean locus of $\ms X$
  is an open substack, and the inclusion map $\clean(\ms X) \to \ms X$
  is quasi-compact.
\end{prop}

Note that the inclusion of an open substack is representable by
definition, so it makes sense to ask whether or not the inclusion
morphism is quasi-compact, as in (3.10) of \cite{l-m-b}.

\begin{proof} 
  This reduces to the following: if $f:G\to X$ is a proper group
  scheme of finite presentation over an affine scheme then the
  subsheaf $T$ of $X$ over which $f$ is an isomorphism is represented
  by a quasi-compact open immersion $U\inj X$. The (big) sheaf $T$ is
  compatible with base change (by definition).  Since $f$ is of finite
  presentation, we may assume that $X$ is the spectrum of a
  finite-type $\Z$-algebra (and, in particular, Noetherian).  The 
  result is then given by Proposition 4.6.7(ii) of \cite{ega31}.
\end{proof}

\begin{remark}
  In the absence of the hypothesis that $G\to X$ is proper, it is easy
  to see that the clean locus need not be open.
\end{remark}

\begin{defn}
  A quasi-algebraic stack $\ms X$ is \emph{bald\/} if it has proper
  diagonal and the clean locus $\clean(\ms X)\subset\ms X$ is
  schematically dense.
\end{defn}

 Recall that an open substack $\iota:\ms U\inj\ms X$ is
\emph{schematically dense\/} if any closed substack $\ms Z \inj \ms X$
which contains $\ms U$ is necessarily equal to $\ms X$.

The remainder of the present paper is devoted to proving the following.

\begin{thm}\label{thm:summary} The following quasi-algebraic stacks
  are isonatural:
\begin{enumerate}
\item bald Artin stacks; 
\item $\B{G}$, where $G$ is a finite \'etale group space
over a locally Noetherian algebraic space;
\item $\B{G}$, where $G$ is an abelian group space locally of finite
  presentation over an algebraic space;
\item Brauer $G$-gerbes over a strongly R1 algebraic space, with 
$G$ a diagonalizable finite group scheme;
\item Brauer $\G_m$-gerbes over an algebraic space.
\end{enumerate}
\end{thm}

Part (1) is treated in Section \ref{sec:anti-gerbes}. Part (2) is
Proposition \ref{sec:classifying-stacks-8}, and we have already proved 
part (3) because we can recover abelian stabilizers functorially by 
Theorem \ref{T:nanas}. We prove part (4) in Proposition 
\ref{P:const-band-recov}, and part (5) in Proposition \ref{sec:g_m-gerbes-1}.

Before embarking on the proof of the theorem, we note that the results
of Section \ref{sec:rec-props} are sufficiently fine to sift out all
the classes of stacks appearing in Theorem \ref{thm:summary} from
their functors.

\begin{prop}\label{P:classes-isonatural} 
  Each of the classes of stacks in Theorem \ref{thm:summary} is
  isonatural.
\end{prop}

We will use the following technical lemma.

\begin{lem}\label{L:dense-flat} Let $X$ be a scheme, and $U$ an open 
  subscheme such that the inclusion $\iota: U \to X$ is quasi-compact.
  Let $f: T \to X$ be a flat morphism. 
  \begin{enumerate}
  \item If $U$ is schematically dense in $X$ then $f^{-1}(U)$ is
    schematically dense in $T$.
  \item If $f$ is faithfully flat, then $U$ is schematically dense in
    $X$ if and only if $f^{-1}(U)$ is schematically dense in $T$.
  \end{enumerate}
\end{lem}

\begin{proof} Because $\iota$ is quasi-compact and separated, we have
  that $\iota_*\mc O_U$ is quasi-coherent on $\mc O_X$, so schematic
  density of $U$ in $X$ is equivalent to injectivity of $\mc O_X \to
  \iota_* \mc O_U$.  Similarly, $f^{-1}(U)$ is schematically dense in
  $T$ if and only if $\mc O_T \to \iota_{T*} \mc O_{f^{-1}(U)}$ is
  injective, where $\iota_T:f^{-1}(U) \to T$ is the natural
  inclusion.  Now suppose $f$ is flat.  The commutativity of pushforward
  with flat base change shows that $\iota_T$ is the pullback of
  $\iota$.  Thus, if $\iota$ is injective then so is $\iota_T$ (as $f$
  is flat).  Moreover, if $f$ is faithfully flat then $\iota$ is
  injective if and only if $\iota_T$ is injective.  This establishes
  the lemma.
\end{proof}

\begin{remark}
  The arguments in the proof of Lemma \ref{L:dense-flat} also show 
  that an open substack $\ms
  U\to\ms X$ of an Artin stack is schematically dense if and only if
  its preimage in any smooth cover of $\ms X$ is schematically dense.
\end{remark}

\begin{proof}[Proof of Proposition \ref{P:classes-isonatural}] 
  We first wish to see that being a bald Artin stack is an isonatural
  property. That being an Artin stack is isonatural follows from
  Corollary \ref{cor:summary-absolute}. Moreover, it is clear from
  Lemma \ref{L:dense-flat} that if $\ms X$ is an Artin stack, then the
  clean locus is schematically dense if and only if for some (and
  hence any) smooth cover $U \to \ms X$ by a scheme $U$, the preimage
  of the clean locus is schematically dense.  But the preimage of the
  clean locus in $U$ can be tested on the level of functors, since by
  Corollary \ref{P:binana-bipoints-1} a point $t:T \to \ms X$ (which
  we will choose to factor through $U$) factors through the clean
  locus if and only if $\Bin(t)$ is isomorphic to the singleton
  functor on the category of $T$-schemes. Thus we conclude that
  baldness is isonatural.

  Next, we recall that Corollary \ref{cor:summary-absolute} tells us
  that we can recognize gerbes over a given algebraic space $X$ from
  their associated functors, with $X$ recovered as the sheafification
  of the functor. Because $X$ is an algebraic space determined by
  $F_{\ms X}$, we can impose any conditions we wish on it, so we see
  for instance that being a gerbe over a strongly R1 algebraic space
  is an isonatural property. We next note that being a neutral gerbe
  is isonatural, since neutrality is equivalent to having a global 
  section.

  We also remark that the gerbe classes (2)-(5) all consist of Artin
  stacks which are locally of finite presentation over $X$. Both of
  these properties are isonatural by Corollary \ref{cor:summary-absolute}
  and Theorem \ref{thm:summary-mors} respectively.

  For (2), we already know that being of the form $\B{G}$ is
  isonatural.  We also know that $G$ is \'etale if and only if any
  section $X\to\ms X$ is \'etale, which is isonatural (by the
  criterion of Proposition \ref{P:formal-crit-affine}, applied to
  \'etale-local affines on $X$).  Finally, $G$ is finite if and only
  if $\ms X$ is separated over $X$ (as then the diagonal of $\B{G}$ --
  whose pullback to $X$ is just $G$ -- is proper), which is isonatural
  by Theorem \ref{thm:summary-mors} because $X$ is assumed to be
  locally Noetherian, and we are assuming that $\ms X$ is
  locally of finite presentation over $X$.

  For (3), we note that if $\ms X$ is a neutral gerbe with global
  section having automorphism group sheaf $G$, then $\ms X \cong \B{G}$.
  Thus, to test whether $\ms X \cong \B{G}$ with $G$ abelian, it 
  suffices to find a global section having abelian automorphism sheaf,
  which is an isonatural property by Theorem \ref{T:nanas}.

  The property that a gerbe $\ms X$ is Brauer is isonatural, since
  $\ms X$ is Brauer if and only if $F_{\ms X}(P)$ is non-empty for
  some Brauer-Severi space $P$ over $X$.

  Finally, we can test whether a given gerbe is a $G$-gerbe for $G$
  diagonalizable or equal to $\G_m$ by first verifying that all
  automorphism groups are abelian, and then using Corollary
  \ref{cor:band-love} to recover the automorphism sheaf.
\end{proof}

Assuming Theorem \ref{thm:summary}, we can finish the proof that
common moduli problems are isonatural.

\begin{proof}[Proof of Theorem \ref{thm:main-examples}] 
  Indeed, the moduli stacks of marked curves are bald (as
  $\widebar{\ms M}_g$ is smooth and geometrically connected
  \cite{d-m}, and the generic curve has trivial automorphism group),
  the Picard stack and the stacks of stable vector bundles are Brauer
  $\G_m$ gerbes (see Remark \ref{R:bass-ackwards} below), while the
  stacks of coherent sheaves with fixed determinant and of $n$th roots
  of an invertible sheaf are Brauer $\m_n$ gerbes over strongly R1
  algebraic spaces (Theorem 9.3.3 and Theorem 9.4.3 of \cite{h-l} for
  stable sheaves, clear for roots of an invertible sheaf).  Thus, all
  cases follow from Theorem \ref{thm:summary}.
\end{proof}

\begin{remark}\label{R:bass-ackwards}
  Let $X\to S$ be a flat projective morphism of finite presentation
  which is cohomologically flat in degree $0$.  The stack $\ms S_r$ of stable
  sheaves of rank $r$ on $X$ is a $\G_m$-gerbe over a separated
  algebraic space $S_r$.  Langer's work \cite{la2} shows that
  the connected components of $\ms S_r$ are quasi-compact.  Choose such a
  component $\Sigma$ of $\ms S_r$, with sheafification (i.e., image in
  $S_r$) $\widebar{\Sigma}$.  Let $\ms V$ be the universal family on
  $X\times\Sigma$.  For sufficiently large $n$ the sheaf
  $(\pr_2)_\ast\ms V(n)$ is a locally free $\Sigma$-twisted sheaf (in
  the terminology of \cite{li1}).  Taking its projectivization yields
  a Brauer-Severi space $P\to\widebar{\Sigma}$ with Brauer class equal
  to the class of $\Sigma$.  This shows that $\ms S_r\to S_r$ is a
  Brauer $\G_m$-gerbe.  Fixing the determinant yields a Brauer $\m_r$-gerbe.
\end{remark}

\subsection{Bald stacks}
\label{sec:anti-gerbes}

We now show that bald Artin stacks are determined by their associated
functors; that is, we prove Theorem \ref{thm:summary}(1).  We make use
of the formalism of groupoids in this section; this is clearly
described in paragraph 2.4.3 of \cite{l-m-b} and section 2 of \cite{k-m}.

\begin{remark}
  There is a mild precursor to the main result of this section.
  Vistoli's proof of Proposition 2.8 of \cite{vi1} implies that a tame
  regular separated Deligne-Mumford stack ${\ms X}$ locally of finite
  type over a locally Noetherian algebraic space with trivial automorphism
  groups in codimension $1$ is determined {\it among stacks of this
    form\/} by its coarse moduli space.  Our more general result
  arises from the extra information made available by the associated
  functor.
\end{remark}

\begin{defn} 
  Let $R\rightrightarrows Z$ be a groupoid object in the category of
  algebraic spaces, $\ms Y$ a stack, and $F$ a functor.  Write
  $p,q:R\to Z$ for the two structure maps.
  \begin{enumerate}
  \item An \emph{$R$-equivariant object of $\ms Y$ over $Z$\/} is a
    pair $(\phi,\alpha)$ with $\phi:Z\to\ms Y$ and $\alpha:\phi
    p\simto\phi q$ an isomorphism of arrows $R\to\ms Y$, such that the
    coboundary $\delta\alpha$ equals $\id$ on $R\times_{p,Z,q} R$.
  \item An \emph{$R$-invariant object of $F$ over $Z$\/} is a
    map $\psi:Z\to F$ such that $\psi p=\psi q:R\to F$.
  \end{enumerate}
\end{defn}

\begin{remark}
  We remind the reader of the definition of the coboundary
  $\delta\alpha$.  The groupoid structure yields three maps $R\times_Z
  R\to R$: the two projections $\pr_1,\pr_2$ and the multiplication
  map $m$.  We then have
  $\delta\alpha=\pr_1^{\ast}(\alpha)m^{\ast}(\alpha)^{-1}\pr_2^{\ast}(\alpha)$.
  Setting the coboundary equal to the identity is the same as
  requiring that the multiplication in $R$ correspond to composition
  of arrows.
\end{remark}

It is clear that any $R$-equivariant object of $\ms Y$ over $Z$ yields
an $R$-invariant object of $F_{\ms Y}$ over $Z$.  We will show that
when $\ms Y$ is bald, there
is an equivalence between these two notions as long as $Z$ dominates the
clean locus and the groupoid structure maps $p$ and $q$
are flat.

Given a groupoid $R\rightrightarrows Z$ as above, let $[Z/R]$ be the
stackification of the category fibered in groupoids whose fiber over
$T$ is the groupoid $R(T)\rightrightarrows Z(T)$.  We make no
assumptions on the structure maps of $R\rightrightarrows Z$, but we
assume that $[Z/R]$ is the stackification in the big fppf site of the
base scheme.

The $R$-equivariant objects of $\ms Y$ over $Z$ form a groupoid, in
which the isomorphisms $(\phi,\alpha)\simto(\phi',\alpha')$ are given
by isomorphisms $\phi\simto\phi'$ which are compatible with $\alpha$
and $\alpha'$ in the obvious way.

A proof of the following proposition may be found in Proposition 3.2 of
\cite{k-l2}.

\begin{prop}\label{P:maps-is-maps}
  The map sending $f:[Z/R]\to\ms Y$ to the associated $R$-equivariant
  object of $\ms Y$ over $Z$ is an equivalence of categories.
\end{prop}

\begin{notn}
  Given $Z,R,\ms Y,F$ as above, we will write $\ms Y_Z^R$ for the set
  of isomorphism classes of $R$-equivariant objects of $\ms Y$ over
  $Z$ and $F_Z^R$ for the set of $R$-invariant objects of $F$ over $Z$.
\end{notn}

It is a standard result (see Corollary 8.1.1, \cite{l-m-b}) that
$\clean(\ms X)$ (see Definition \ref{D:cl} above for the definition of
$\clean$) is isomorphic to an algebraic space when $\ms X$ is an Artin
stack. Thus, suppose $\ms X$ is a quasi-algebraic stack and $\ms
U\inj\ms X$ is a quasi-compact open immersion from an algebraic space.

\begin{prop}\label{P:equivariant-invariant-comp}
  Given a flat groupoid of algebraic spaces $R\rightrightarrows Z$,
  the functor $F$ defines a bijection between the set of elements
  $[(\phi,\alpha)]\in\ms X_Z^R$ such that $\phi^{-1}(\ms U)\subseteq Z$
  is schematically dense and the set of elements $\psi\in (F_{\ms
    X})_Z^R$ such that $\psi^{-1}(F_{\ms U})\subseteq Z$ is
  schematically dense.
\end{prop}
\begin{proof}
  Given an $R$-invariant object $\psi$ of $F_{\ms X}$ over $Z$, choose
  a lift to $\phi:Z\to\ms X$.  Since $\psi$ is $R$-invariant, there is
  some isomorphism $\alpha:\phi p\simto\phi q$.  If $\psi^{-1}(F_{\ms
    U})$ is schematically dense in $Z$, then because $R\times_Z R\to
  Z$ is flat, it follows from Lemma \ref{L:dense-flat} that 
  $W:=p^{-1}\psi^{-1}(F_{\ms U})$ is schematically dense.
  Since the inertia stack of $\ms U$ is trivial, we
  thus see that $\delta\alpha|_W=\id$, and since the diagonal of $\ms
  X$ is separated this implies that $\delta\alpha=\id$ on $R\times_Z
  R$.  Thus, the natural map is surjective.  To see that it is
  injective, suppose $(\phi_i,\alpha_i)$, $i=1,2$ are $R$-equivariant
  objects of $\ms X$ over $Z$ such that $\phi_i^{-1}(\ms U)$ is
  schematically dense.  If their images in $(F_{\ms X})_Z^R$ are
  equal, there is some isomorphism $\beta:\phi_1\simto\phi_2$.  It
  follows immediately from the schematic density hypothesis that the
  diagram
$$\xymatrix{\phi_1 p\ar[r]^{\alpha_1}\ar[d] & \phi_1 q\ar[d] \\
\phi_2 p\ar[r]^{\alpha_2} & \phi_2 q}$$
commutes, which means that $\beta$ is an isomorphism of
$R$-equivariant objects.
\end{proof}

\begin{proof}[Proof of Theorem \ref{thm:summary}(1)]
Given two bald Artin stacks $\ms X_1$ and $\ms X_2$, let
$R_i\rightrightarrows Z_i$ be a smooth presentation of $\ms X_i$,
$i=1,2$.  Given an isomorphism $\phi:F_{\ms X_1}\simto F_{\ms X_2}$,
there results an $R_1$-invariant object of $F_{\ms X_2}$ over $Z_1$.
Moreover, this object must come from a smooth surjection $Z_1\to\ms
X_2$.  Applying Proposition \ref{P:equivariant-invariant-comp}, there results a
unique map (up to isomorphism) $\psi:\ms X_1\to\ms X_2$ giving rise to
$\phi$.  Reversing the roles of $1$ and $2$ yields $\widebar\psi:\ms
X_2\to\ms X_1$ inducing $\phi^{-1}$.  The composition
$\widebar\psi\psi:\ms X_1\to\ms X_1$ corresponds to the
$R_1$-invariant object of $F_{\ms X_1}$ over $Z_1$, hence must be an
isomorphism.  Reversing the roles of $1$ and $2$ again, we see that
$\psi$ is an isomorphism.  That the association $\phi\mapsto\psi$
yields a retraction $\Isom(F_{\ms X_1},F_{\ms X_2})\to\Isom(\ms
X_1,\ms X_2)$ follows from the uniqueness in
Proposition \ref{P:equivariant-invariant-comp} and is left to the reader.
\end{proof}

\subsection{Classifying stacks}
\label{sec:class}

In this section we show that given an algebraic space $X$ and a finite
\'etale group space $G\to X$, the group $G$ is uniquely determined up
to inner forms by the functor associated to the classifying stack
$\B{G}$.  (For the reader familiar with Giraud's terminology, this
says that the functor associated to $\B{G}$ uniquely determines the
isomorphism class of the band associated to $G$.)  This will ultimate
show that classifying stacks for finite \'etale group spaces are
isonatural.

We will recover $\B{G}$ by (in essence) recovering the functor of points of
$G$ (in the stack of bands) from a subcategory of pointed
schemes.  This subcategory arises from a functorial construction for a
pointed scheme with a given finite fundamental group.
(There are of course subleties associated to doing this over
schemes which are not geometric points.)  We thus begin with some
results pertaining to the \'etale fundamental group.

\begin{lem}\label{L:section-lemma}
  Suppose $X$ is an algebraic space, $Y_1,Y_2\to X$ are finite 
  \'etale morphisms, and $f:P\to X$ is a faithfully flat morphism
  with geometrically connected fibers.  Pullback induces a bijection
  between $X$-morphisms $Y_1\to Y_2$ and $P$-morphisms $Y_1\times_X P\to
  Y_2\times_X P$.
\end{lem}
\begin{proof}
  First suppose $Y_1\to X$ and $Y_2\to X$ are disjoint unions of
  copies of $X$ and $X$ is connected.  Since the fibers of $P\to X$
  are geometrically connected, it follows that any $X$-morphism
  $(Y_1)_P\to Y_2$ factors through a morphism $Y_1\to Y_2$.  Since
  $Y_1\to X$ and $Y_2\to X$ have this form \'etale locally on $X$, we
  see that the natural map of \'etale sheaves
  $\chi:\shom_X(Y_1,Y_2)\to f_{\ast}\shom_P((Y_1)_P,(Y_2)_P)$ is an
  epimorphism.  On the other hand, since $P\to X$ is faithfully flat,
  $\chi$ is also a monomorphism (cf.\ Theorem VIII.5.2 of
  \cite{sga1}).  Thus, $\chi$ is an isomorphism, and the result
  follows.
\end{proof}
\begin{remark}
  Lemma \ref{L:section-lemma} applies notably when $Y_1=X$, and thus
  to sections of a given finite \'etale covering.
\end{remark}

\begin{prop}\label{L:fund-gp-lem}
  Suppose $Z$ is a normal connected algebraic space and
  $Y\to Z$ is a smooth surjective map of finite presentation between
  algebraic spaces with connected geometric fibers.  Let
  $\ast:\Spec\kappa\to Y$ be a geometric point over the generic point
  $\theta$ of $Z$.  Then the natural sequence of groups
$$\pi_1(Y_{\widebar\theta},\ast)\to\pi_1(Y,\ast)\to\pi_1(Z,\ast)\to 1$$
is exact.
\end{prop}
\begin{proof} We may first reduce to the case that $Z$ is excellent,
  using standard limiting arguments.  This statement is known when $Z$
  is the spectrum of a field (Theorem IX.6.1 of \cite{sga1}).
  Thus, as $\pi_1(Y_{\theta},\ast)\to\pi_1(Y,\ast)$ is surjective
  (Proposition V.8.2 of [{\it ibid\/}.]), it follows that the left-hand
  map in the sequence has normal image.  To show that the sequence is
  exact in the middle, it suffices to show that any Galois cover $W\to
  Y$ which is trivial on the geometric generic fiber comes by pullback from
  $Z$.  Since this is already known over the function field of $Z$, we
  find a field extension $L/\kappa(Z)$ such that the normalization of
  $Y_{\theta}$ in $L|_{Y_{\theta}}$ is isomorphic to $W_{\theta}$.  Since
  $W_{\theta}$ is the generic fiber of $W$ and $Z$ (and thus $Y$) is
  normal, we see that the normalization of $Y$ in $L_{Y_{\theta}}$ is
  isomorphic to $W$.  But since $Y\to Z$ is smooth, this is just the
  pullback of the normalization of $Z$ in $L$.  Writing $Z'\to Z$ for
  this normalization, we thus have that $Z'\times_Z Y\to Y$ is
  \'etale, from which it follows that $Z'\to Z$ is \'etale.  Thus,
  $W\to Y$ is the pullback of an \'etale (in fact, Galois) covering of
  $Z$.  It follows from Proposition V.6.11 of [{\it ibid\/}.] that the
  sequence is exact in the middle.  Exactness on the right follows
  from the fact that the fibers of $Y$ are geometrically connected
  (Corollary IX.5.6 of [{\it ibid\/}.]).  (Cf.\ the proof of
  Theorem X.1.3 and Corollary X.1.4 of [{\it ibid\/}.].)
\end{proof}

\begin{cor}\label{sec:classifying-stacks-4}
  Let $f:Y\to X$ be a smooth surjective morphism of finite presentation between
  connected algebraic spaces with geometrically connected fibers.  Let
  $y\to Y$ be a geometric point of $Y$.  If $\pi_1(Y_{f(y)},f(y))=0$ then
  the natural map $\pi_1(Y,y)\to\pi_1(X,f(y))$ is an isomorphism.
\end{cor}
\begin{proof}
  By standard methods, we may assume that $X$ is excellent and
  Noetherian; thus, the normalization $X'\to X$ is a finite surjective
  morphism of finite presentation.  Let $W\to Y$ be a finite \'etale
  morphism.  By Theorem IX.4.7 of
  \cite{sga1}, we know that $X'\to X$ is a morphism of effective
  descent for finite \'etale covers.  Applying Proposition
  \ref{L:fund-gp-lem}, we see that there is a finite \'etale cover
  $W'\to X'$ and an isomorphism $W\times_X X'\simto W'\times_{X}Y$
  over $Y\times_X X'$.  The descent datum on $W$ thus gives rise to a
  descent datum on $W'\times_X Y$ (with respect to the morphism
  $Y\times_X X'\to Y$). By Lemma \ref{L:section-lemma}, there is an
  induced descent datum on $W'$ (with respect to the morphism $X'\to
  X$).  Since $X'\to X$ is effective, there is a finite \'etale
  covering $U\to X$ giving rise to the descent datum on $W'$.
  Applying Lemma \ref{L:section-lemma} once more, we see that
  $U\times_X Y$ and $W\times_X X'$ are isomorphic via an isomorphism preserving
  the descent data.  Applying the effectivity to $Y\times_X X'\to Y$
  once again, we see that there is an isomorphism $U\times_X Y\simto
  W$.

  This shows that the Galois categories of finite \'etale covers of
  $Y$ and $X$ (with the fiber functors induced by the given points)
  are equivalent, which implies that the fundamental groups are
  naturally isomorphic.
\end{proof}

\begin{prop}\label{sec:classifying-stacks-1}
  Suppose $f:Y\to X$ is a smooth surjective morphism of finite
  presentation with connected geometric fibers between connected
  algebraic spaces.  Suppose $y\to Y$
  is a geometric point.  Let $G$ be a finite group with a free
  $X$-action on $Y$.  If $\pi_1(Y_{f(y)},y)=0$, then there is a
  natural isomorphism $\pi_1(Y/G,y)\simto\pi_1(X,f(y))\times G$.
\end{prop}
\begin{proof}
  The Galois covering $Y\to Y/G$ induces an exact sequence
$$1\to\pi_1(Y,y)\to\pi_1(Y/G,y)\to G\to 1.$$
By Corollary \ref{sec:classifying-stacks-4}, the natural morphism
$\pi_1(Y,y)\to\pi_1(X,f(y))$ is an isomorphism.  But then the natural
map $\pi_1(Y/G,y)\to\pi_1(X,f(y))$ yields a splitting of the left-hand
map of the exact sequence, which yields a splitting
$\pi_1(Y/G,y)\simto\pi_1(Y,y)\times G$, as required.
\end{proof}

Given a scheme $S$, the category $S\Sch_\bullet$ of \emph{pointed
  $S$-schemes\/} is the category of $S$-schemes under $S$ (i.e.,
arrows $S\to X$ in the category of $S$-schemes).  Any functor 
$F:S\Sch^\circ\to\Set$ naturally yields a functor
$F_\bullet:S\Sch_\bullet^\circ\to F(S)\textrm{-}\mathbf{Set}$.  We will
freely use the associated functors $F_\bullet$ in studying
reconstruction of stacks; since these associated functors arise
abstractly from the original functor $F$, no additional information is
introduced in their formation.

\begin{prop}\label{sec:classifying-stacks}
  There is a functor $\beta:\fingps\to\Z\Sch_{\bullet}$ such that for
  any connected algebraic space $X$ with geometric generic
  point $\widebar{\theta}:\spec\widebar{\kappa}\to X$, there is an
  isomorphism
  $\pi_1(X\times\beta(G),\widebar{\theta})\simto\pi_1(X,\widebar{\theta})\times
  G$ which is functorial in $G$.
\end{prop}

To construct $\beta(G)$, we will make a functorial pointed
quasi-projective $\Z$-scheme with a (functorial) free $G$-action.
This comes in a straightforward way from the regular representation of
$G$.

Given a finite group $G$, let $V(G)$ be the geometric vector bundle
associated to the regular representation
with its functorial $A$-basis $\{e_g\}_{g\in G}$.  Fix a positive
integer $N\geq 2$ and let $W(G)=\Hom(A^N,V(G))$, where $A$ is the
trivial representation with basis $1$ (so that $A^N$ has a natural
basis $b_1,\ldots,b_N$).  Let $\ast$ be the $A$-point of $W(G)$
consisting of the map sending each $b_i$ to $e_1$.  The formation of
$V(G)$, $W(G)$, and $\ast$ is clearly functorial in $G$ and $A$.  Let
$P(G)\to\spec A$ denote the projectivization of $W(G)$.  The action of
$G$ on $W(G)$ induces an action on $P(G)$, and the $A$-point $\ast$ of
$W(G)$ gives
rise to a natural section (which we will also denote $\ast$) of
$P(G)\to\spec A$.

\begin{prop}\label{sec:classifying-stacks-3}
  There is an open subscheme $U(G)\subset P(G)$ such that 
  \begin{enumerate}
  \item $\ast$ is contained in $U(G)$;
  \item $U(G)\to\spec A$ is surjective;
  \item for each geometric point $x\to\spec A$, the inclusion
    $U(G)_x\subset P(G)_x$ is the complement of a union of hyperplanes
    of codimension at least $2$;
  \item the action of $G$ on $U(G)$ is free;
  \item given a map $\epsilon:G\to G'$ of finite groups, the induced map
    $W(G)\to W(G')$ induces a map $U(G)\to U(G')$ of pointed
    $A$-schemes which is $\epsilon$-equivariant.
  \end{enumerate}
\end{prop}

\begin{proof}[Proof of Proposition \ref{sec:classifying-stacks}, given
  Proposition \ref{sec:classifying-stacks-3}]
  The proof applies immediately by applying Proposition
  \ref{sec:classifying-stacks-1} to the family $U(G)\times_{\spec\Z}X$
  given by Proposition \ref{sec:classifying-stacks-3} (when $A=\Z$).
\end{proof}

We now give a proof of Proposition \ref{sec:classifying-stacks-3}.

\begin{lem}\label{sec:classifying-stacks-2}
  Let $U'(G)\subset P(G)$ be the largest open subscheme such that for
  any map of finite groups $G\to G'$, the rational morphism
  $P(G)\dashrightarrow P(G')$ is regular on $U'(G)$.  Then $U'(G)$ is
  the complement of a union of flat families of linear subspaces of
  $P(G)$ of codimension at least $2$.
\end{lem}
\begin{proof}
  Any map $G\to G'$ factors through a quotient group $G\to\widebar G$.
  The kernel of the induced map $W(G)\to W(\widebar G)$ is a linear
  space of codimension at least $N$ (as $W(\{1\})$ has dimension
  $N$).  For each such quotient $G\to\widebar G$ there is a subbundle
  $W_{\widebar G}\subset W(G)$ parametrizing the family of kernels;
  removing the union of the corresponding subspaces of $P(G)$ yields
  $U'(G)$, as desired.
\end{proof}

Define a closed subscheme $Z(G)\subset U'(G)$ by taking the
scheme-theoretic union of all preimages under (surjective) quotient morphisms
$U'(G)\to U'(\widebar G)$ of all fixed loci for the action of
non-identity elements of $\widebar G$.

\begin{lem}\label{sec:classifying-stacks-5}
 With the immediately preceding notation, the closed subscheme
 $Z(G)\subset U'(G)$ has codimension $2$ in every geometric fiber of
 $U'(G)\to\spec A$.  Moreover, $\ast$ factors through $U'(G)\setminus
 Z(G)$.  Finally, for every map of finite groups $G\to G'$, the
 induced map $U'(G)\to U'(G')$ induces a map $U'(G)\setminus Z(G)\to
 U'(G')\setminus Z(G')$.
\end{lem}

The proof of Lemma \ref{sec:classifying-stacks-5} requires a bit of
analysis of the eigenvectors for the elements of $G$ acting on $W(G)$.

\begin{notn}
  Given a linear representation $R$ of $G$, a scalar $\lambda\in k$,
  and an element $g\in G$, let $R^{g,\lambda}$ denote the submodule of
  $R$ on which $g$ acts as multiplication by $\lambda$.
\end{notn}

If $g$ has order $\nu$, it is clear that $R^{g,\lambda}$ is $0$ if 
$\lambda$ is not a $\nu$th root of unity.  

\begin{lem}
  With the above notation, let $\nu$ be the order of $g\in G$.  Assume
  $g\neq 1$ (for the sake of non-stupidity).  For any
  $\lambda\in\m_{\nu}(k)$, the submodule $V(G)^{g,\lambda}$ is a
  locally direct summand of $V(G)$ of corank $|G|(1-1/\nu)$.  The submodule
  $W(G)^{g,\lambda}$ is a locally direct summand of $W(G)$ of corank 
  $N|G|(1-1/\nu)\geq N\geq 2$.
\end{lem}
\begin{proof}
  The group decomposes into $|G|/\nu$ left orbits for the action of
  $g$ of size exactly $\nu$.  Choosing an element $h_i$ in each orbit,
  we see that for an element $\sum\alpha_he_h\in V(G)^{g,\lambda}$,
  the coefficient of $e_{g^sh_i}$ must be $\lambda^s\alpha_{h_i}$.  Thus,
  each element of $V(G)^{g,\lambda}$ is uniquely determined by the set
  of coordinates $\alpha_{h_i}$, $i=1,\ldots,|G|/\nu$.  This gives the
  statement for $V(G)$, and the assertion on $W(G)$ follows.
\end{proof}

Let $f:G\to G'$ be a homomorphism of finite groups.  There are induced
maps $f_\ast:V(G)\to V(G')$ and $f_{\ast}:W(G)\to W(G')$ of
$A$-modules which are equivariant over $f$ in the standard sense.

\begin{lem}\label{L:fixed-pts-lem}
Given a non-zero $\alpha\in V(G)$, suppose there is some $g'\in G'$ and
$\lambda\in A$ such that
$f_\ast\alpha\in V(G')^{g',\lambda}$.  Then $g'\in f(G)$ and
$f_\ast\alpha\in\iota_\ast V(f(G))^{g',\lambda}$, where $\iota:f(G)\to G'$ is
the natural inclusion.
\end{lem}
\begin{proof}
  Write $\alpha=\sum\alpha_he_h$.  By assumption, for all $h\in G$
  such that $\alpha_h\neq 0$ we have that $g'f(h)\in f(G)$.  Since
  there is some $h$ with $\alpha_h\neq 0$, we see that $g'\in f(G)$.
  The lemma follows immediately.
\end{proof}

\begin{proof}[Proof of Lemma \ref{sec:classifying-stacks-5}]
  Since the closed subset underlying $Z(G)$ is compatible with base
  change, it suffices to prove the lemma assuming that $A=k$ is an
  algebraically closed field.  A fixed point for the action of $g$ on
  $U'(G)$ is the image of an eigenvector in $W(G)$.  Given a quotient
  $\widebar G$, an element $\bar g\in\widebar G\setminus\{1\}$, and a
  scalar $\lambda\in A$, define $W(G)^{\bar g,\lambda}$ to be the
  preimage of $W(\widebar G)^{\bar g,\lambda}$ under the natural
  surjection $W(G)\to W(\widebar G)$.  Considering all non-trivial
  quotients of $G$ at once, we define
$$W_0(G)=
W(G)\setminus\left(\bigcup_{\overset{G\surj \widebar G\neq\{1\}}
  {\bar g\in\widebar G\setminus\{1\},\lambda\in k}}
W(G)^{\bar g,\lambda}\right).$$ It is easy to see that $W_0(G)$ is the
complement of the union of finitely many (locally direct summand)
vector subbundles of $W(G)$ of codimension at least $N$ and that
$\ast\in W_0(G)$.  Thus, $W_0(G)$ is an open cone in $W(G)$ whose
complement has codimension at least $2$.  Moreover, the image of
$W_0(G)$ in $P(G)$ is precisely $U'(G)\setminus Z(G)$, as desired.

Applying Lemma \ref{L:fixed-pts-lem}, we see that given a map $G\to H$, the
induced map $W(G)\to W(H)$ sends $\ast$ to $\ast$ and $W_0(G)$ into
$W_0(H)$, yielding an induced pointed map $\widebar W_0(G)\to\widebar
W_0(H)$.  This gives the final functoriality statement of Lemma
\ref{sec:classifying-stacks-5}
\end{proof}

\begin{proof}[Proof of Proposition \ref{sec:classifying-stacks-3}]
  Using the notation of Lemma \ref{sec:classifying-stacks-5}, setting
  $U(G)=U'(G)\setminus Z(G)$ yields a functorial open subscheme with a
  free action (as all fixed loci have been removed) which is the
  complement of a union of linear subspaces of codimension at least
  $2$ in every fiber.
\end{proof}

\begin{remark}\label{R:linear-ok-char-0}
 The reader will note that we could have avoided the use of both the
 projective space $P(G)$ and the eigenspaces
 $W(G)^{g,\lambda}$ in characteristic $0$.  In that case, since $W(G)$
 is itself simply connected, it suffices to simply remove the fixed
 point loci directly and take the quotient by $G$.  In this guise, our
 construction looks more similar to that which arises in the study of
 equivariant cohomology, as in \cite{e-g2}. 
\end{remark}

We next recall a few facts about bands which will be useful in the sequel.
The reader is referred to Chapter IV of \cite{gi4} for the definitive 
treatment of the subject (and further context).

Given a site $S$ (the reader may think of the Zariski or \'etale site
of a scheme), the stack of bands is defined as the stackification of a
quotient of the stack of groups as follows: Given two sheaves of
groups $G$ and $H$ over an object $T$ of $S$, there is natural right
action of $\saut(G)$ (resp.\ left action of $\saut(H)$) on
$\Isom(G,H)$.  Define a new fibered category $\widebar B$
over $S$ by taking as objects over $T$ the set of sheaves of groups on
$T$, but with homomorphism sheaf
$\shom_B(G,H)=\saut(H)\setminus\shom(G,H)/\saut(G)$.  The stack
$\bands_S$ of bands on $S$ is then defined to be the stackification of
$\widebar B$.  

\begin{lem}\label{L:band-lem}
  Given an object $T$ of $S$, a sheaf of groups $G$ on $T$, and an
  inner form $G'$ of $G$, there is an isomorphism $G\simto G'$ in the
  category of bands.
\end{lem}
\begin{proof}
  Since $G'$ is an inner form of $G$, there is a covering $U\to T$ and
  an isomorphism $\phi:G|_U\simto G'|_U$ whose coboundary, viewed as
  an automorphism of $G|_{U\times_T U}$, is conjugation by a section
  of $G$.  But any such automorphism is trivial in the category of
  bands, so $\phi$ descends to an isomorphism $G\to G'$ in
  $\bands_S(T)$.
\end{proof}

\begin{lem}\label{L:iso-band-inner}
  Suppose $G$ and $H$ are sheaves of groups on $T$.  If $\phi:G\cong H$ in
  $\bands_S(T)$ then there is an inner form $H'$ of $H$ and an
  isomorphism of sheaves of groups $\psi:G\to H'$
\end{lem}
\begin{proof}
  There is a covering $U\to T$ and an isomorphism $\alpha:G|_U\simto
  H|_U$ whose coboundary on $U\times_T U$, viewed as an automorphism
  of $H$, is conjugation by a section $\sigma\in H(U\times_T U)$.
  Moreover, it is formal that $\sigma$ satisfies the $1$-cocycle
  condition, and thus yields an inner form $H'$ of $H$.  Composing
  with the natural isomorphism $H|_U\simto H'|_U$ yields an
  isomorphism $G|_U\to H'|_U$ with trivial coboundary, yielding the
  result.
\end{proof}

When $S$ is the punctual site (e.g., the small \'etale site of a separably
closed field), the stack of bands is just the quotient category of the
category of groups which replaces $\Isom(G,H)$ with the set of
conjugacy classes of such isomorphisms.

\begin{prop}\label{P:fund-bands}
  Let $X$ be a Galois category with fiber functor $\ast$.  For any
  finite group $G$, there is a natural isomorphism between
  $\Hom_{\bands}(\pi_1(X,\ast),G)$ and the set of isomorphism classes
  of (right) $G$-torsors over the final
  object of $X$.
\end{prop}
\begin{proof}
  Given a $G$-torsor $T$ in the category of $\pi$-sets, the choice of
  a point $t\in T$ yields a homomorphism $\pi\to G$ which sends
  $\alpha$ in $\pi$ to $g$ in $G$ such that $\alpha t=tg$.  Changing
  the choice of $t$ changes the map by an inner automorphism.
  Conversely, given such a map, one gets a left action of $\pi$ on the
  underlying set of $G$ which commutes with the natural right
  $G$-action.
\end{proof}

Note that the functor $F_{\B{G}}$ (over a space $X$) is naturally
pointed by the isomorphism class of the trivial torsor; we will use
$\ast$ to denote the canonical point.  (The reader with logical qualms
should note that the fact that a stack has the form $\B{G}$ with $G$ a
finite \'etale group space can be detected from the functor by
Proposition \ref{P:classes-isonatural}, and thus the pointing is
isonatural.)  There is a natural subfunctor $F^{\ast}_{\B{G}}$ on
$X\Sch_{\bullet}$ whose value on $\sigma:X\to Y$ (where $Y$ is an
$X$-scheme and $\sigma$ is a section of the structure morphism) is the
preimage of $\ast$ under the restriction map $F_{\B{G}}(Y)\to
F_{\B{G}}(X)$.

\begin{defn}\label{sec:classifying-stacks-7}
  An isomorphism $\psi:F_{\B{G}}\to F_{\B{H}}$ is \emph{pointed\/} if
  $\psi$ sends the isomorphism class of the trivial torsor to the
  isomorphism class of the trivial torsor.
\end{defn}

We will write $\isom_{\ast}(F_{\B{G}},F_{\B{H}})$ for the subgroup of
pointed isomorphisms.  It is clear that any pointed isomorphism
$F_{\B{G}}\to F_{\B{H}}$ induces an isomorphism $F^{\ast}_{\B{G}}\to F^{\ast}_{\B{H}}$.

\begin{lem}\label{sec:classifying-stacks-6}
  Let $X$ be an algebraic space.  Given finite groups $G$ and $H$,
  there is a map 
$$\isom_{\ast}(F_{\B{G}},F_{\B{H}})\to\isom_{\bands}(G,H)$$
such that the composition
$\isom(G,H)\to\isom_{\ast}(F_{\B{G}},F_{\B{H}})\to\isom_{\bands}(G,H)$
is the natural map.
\end{lem}
\begin{proof}
  We may assume without loss of generality that $X$ is connected.  The
  functor $\beta:\fingps\to X\Sch_{\bullet}$ yields a subcategory of
  $X\Sch_{\bullet}$ which is equivalent to the category of finite
  groups.  Moreover, for any finite group $\Gamma$, we have by
  Proposition \ref{P:fund-bands} that
  $F^{\ast}_{\B{G}}(\beta(\Gamma))\cong\Hom_{\bands_X}(\underline
  \Gamma,\underline{G})$.  The result thus follows from the Yoneda
  lemma (applied to the subcategory of bands associated to constant groups).
\end{proof}

\begin{lem}\label{sec:classifying-stacks-9}
  Given a section $a\in F_{\B{H}}(X)$, there is an inner form $H'$ of
  $H$ and an isomorphism $F_{\B{H}}\simto F_{\B{H}}$ carrying $\ast$
  to $a$.
\end{lem}
\begin{proof}
  If $T$ is an $H$-torsor with isomorphism class $a$, it is standard
  that $H'=\aut_H(T)$ is an inner form of $H$.  Sending an $H$-torsor
  $S$ to the $H'$-torsor $\isom(S,T)$ gives the isomorphism in question.
\end{proof}

\begin{prop}\label{sec:classifying-stacks-8}
  Let $G$ and $H$ be finite \'etale group spaces over an 
  algebraic space $X$.  There is a map 
$$\isom(F_{\B{G}},F_{\B{H}})\to\isom_{\bands}(G,H)$$
whose composition with the natural map
$\isom(G,H)\to\isom(F_{\B{G}},F_{\B{H}})$ is the natural map.
\end{prop}
\begin{proof} There is an \'etale surjection $U\to X$ such that
\begin{enumerate}
\item there are finite groups $\widebar G$ and $\widebar H$ with
  isomorphisms $G_U\simto\underline{\widebar G}_U$ and
  $H_U\simto\underline{\widebar H}_U$;
\item the restriction of $\psi$ to the category of $U$-schemes is
  pointed.
\end{enumerate}
There is a resulting diagram of isomorphisms of functors on $U\Sch$
$$\xymatrix{{F_{\B{G}}}_U\ar[r]\ar[d] & {F_{\B{H}}}_U\ar[d]\\
  {F_{\B{\underline{\widebar G}}}}_U\ar[r] &
  {F_{\B{\underline{\widebar H}}}}_U.  }$$ The isomorphism
$G_U\simto\underline{\widebar G}_U$ induces a pointed automorphism of
$F_{\B{\underline{\widebar G}}}|_{U\times_X U}$ whose image in
$\isom_{\bands_{U\times_X U}}(\underline{\widebar
  G},\underline{\widebar G})$ is the descent datum for $G$ (as a form
of $\underline{\widebar G}$); there is a similar automorphism of
$F_{\B{\underline{\widebar H}}}|_{U\times_X U}$.  A straightforward
(but somewhat laborious) diagram chase, starting with the global
isomorphism $F_{\B{G}}\simto F_{\B{H}}$, shows that the lower horizontal
arrow in the above diagram respects the descent data on both sides.

Applying Lemma \ref{sec:classifying-stacks-9}, we thus find an
isomorphism $\underline{\widebar G}_U\simto\underline{\widebar H}_U$
which is compatible with the descent data for $G$ and $H$ (in the
stack of bands).  This gives the desired map 
$\isom(F_{\B{G}},F_{\B{H}})\to\isom_{\bands}(G,H)$.
\end{proof}

\subsection{Gerbes with finite diagonalizable bands}
\label{sec:recov-abel-gerb}

Having treated classifying stacks, the next natural class of stacks to
consider is more general non-neutral gerbes. In this section we will
show Theorem \ref{thm:summary}(4): if $D$ is diagonalizable then any
Brauer $D$-gerbe is isonatural.

Let $X$ be an algebraic space and $A$ an abelian sheaf on
$X$.  In this section, the phrase ``$D$ is a diagonalizable finite
group scheme'' will mean that $D$ is isomorphic to a finite direct sum
of the form $\bigoplus \m_n$. Given an element $g$ in a group $G$,
write $\langle g\rangle\subseteq G$ for the cyclic subgroup generated
by $g$.

\begin{defn}
  Given an integer $n$, the \emph{cohomology presheaf (of degree $n$
    associated to $A$)\/} is the presheaf $\ms H^n(A)$ on $X$-schemes
  such that $\ms H^n(A)(Y\to X)=\H^n(Y,A_Y)$.
\end{defn}
By common abuse of notation, we will often write simply $\ms
H^n(A)(Y)$, the $X$-structure on $Y$ being implicit.

\begin{defn}\label{D:van-set}
  With the above notation, given a cohomology class
  $\alpha\in\H^n(X,A)$, the \emph{vanishing set of $\alpha$\/} is the
  set of arrows $T\to X$ such that $\alpha\in\ker(\ms H^n(A)(X)\to\ms
  H^n(A)(T))$.  The complement of the vanishing set of $\alpha$ is the
  \emph{support of $\alpha$\/}.  Two classes $\alpha$ and $\beta$ have
  \emph{the same support\/} if their supports are equal.
\end{defn}

The sheafification of the cohomology presheaf is well-known to vanish
for $n>0$ (see for example Proposition 2.5 of Chapter II of
\cite{artin-gr-top}).  The motivating question for this section is
``How much information about a cohomology class can we recover from
its support in the cohomology presheaf?''  A few moments of thought
will convince the reader that the best one can hope to do is recover
the cyclic subgroup generated by the class, and this only in the case
of a cyclic coefficient sheaf.  As we will show, in various cases this
actually works.  However, the methods we employ are specific to the
sheaves and (low) cohomological degrees in question.  Further
investigation of this question seems potentially interesting.

\begin{lem}\label{L:pic-mod}
  Let $X$ be a strongly R1 algebraic space and let $\ms
  L$ be an invertible sheaf on $X$.  Let $\mathbf V(\ms L)=\rspec\sym^\ast\ms
  L$ be the geometric line bundle associated to $\ms L$.  Let
  $Z\subseteq\mathbf V(\ms L)$ be the $0$ section and let $\mathbf
  V(\ms L)^\ast=\mathbf V(\ms L)\setminus Z$.
  Then $\Pic \mathbf V(\ms L)^\ast$ is identified via pullback with
  $\Pic(X)/\langle\ms L\rangle$.
\end{lem}
\begin{proof}
  The hypothesis on $X$ allows us to work with Weil divisor classes.
  It is well-known (with the same proof as Proposition II.6.6 of
  \cite{ha1}) that pullback induces an isomorphism
  $\Pic(X)\to\Pic(\mathbf V(\ms L))$.  On the other hand,
  $Z\subseteq\mathbf V(\ms L)$ is an irreducible divisor, so
  $\Pic(\mathbf V(\ms L)^{\ast})$ is isomorphic to $\Pic(\mathbf V(\ms
  L))/\langle \ms O(Z)\rangle$.  It remains to show that $\ms
  O(Z)|_Z\cong\ms L^{\vee}$ (via the natural identification of $Z$
  with $X$).  To compute this, note that $\ms O(-Z)$ is equal to
  $\bigoplus_{i>0}\ms L^{\tensor i}$.  Restricting this to $Z$ is the
  same as tensoring with $\bigoplus_{i\geq 0}\ms L^{\tensor
    i}/\bigoplus_{i>0}\ms L^{\tensor i}$.  This simply divides out by
  $\bigoplus_{i>1}\ms L^{\tensor i}$, and thus we see that $\ms
  O(-Z)|_Z\cong\ms L$, as required.
\end{proof}

\begin{cor}\label{C:recover-line-bdls}
  Suppose $X$ is a strongly R1 algebraic space. If $L_1$
  and $L_2$ are invertible sheaves whose classes in
  $\H^1(X_{\et},\G_m)$ have the same support then $\langle
  L_1\rangle=\langle L_2\rangle$.  
\end{cor}
\begin{proof}
  Pulling back to $\mathbf V(L_1)^\ast$ and using Lemma \ref{L:pic-mod} and
  the support hypothesis, we see that $L_2\in\langle L_1\rangle$.
  Reversing the roles of $L_1$ and $L_2$ shows that $L_1\in\langle
  L_2\rangle$.  The result follows.
\end{proof}

\begin{remark}
  Note that it is essential that the support of the cohomology classes
  be considered on the entire category of schemes and not merely on
  e.g.\ Zariski open subsets.  An example is provided by $\ms O(1)$
  and $\ms O(2)$ on $\P^1$.  Any scheme mapping to $\P^1$ whose image
  excludes a single point will trivialize both $\ms O(1)$ and $\ms
  O(2)$.  Only by considering surjective morphisms to $\P^1$ from
  larger (connected) schemes can we hope to recover enough information
  from the support.
\end{remark}

\begin{lem}\label{L:br-sev-sch-leray}
  If $f:P\to X$ is a Brauer-Severi space then the pullback map
  $\H^2(X,\m_n)\to\H^2(P,\m_n)$ is injective for all $n$.  The kernel
  of the map $\Br(X)\to\Br(P)$ is the cyclic subgroup generated by the
  Brauer class of $P$.
\end{lem}
\begin{proof}
  The first statement follows from the fppf Leray spectral sequence for
  $\m_n$ combined with the fact that $\R^1f_{\ast}\m_n=0$.  (By the
  flat Kummer sequence, we know that the latter sheaf is isomorphic to
  the $n$-torsion subspace of the relative Picard space of $P$ over
  $X$, hence vanishes.)  The second statement comes from the Leray
  spectral sequence applied to $\G_m$, and may be found in Theorem 2
  of Part 2 of Chapter II of \cite{gabber-thesis} (p.\ 193).
\end{proof}

\begin{lem}\label{L:cyclic-gp-muffin}
  Let $X$ be a strongly R1 algebraic space.  If
  $\alpha,\beta\in\H^2(X,\m_n)$ are two Brauer cohomology classes with the
  same support then
  $\langle\alpha\rangle=\langle\beta\rangle\subseteq\H^2(X,\m_n)$.
\end{lem}
\begin{proof}
  Since $\Br(X)=\Br'(X)$, there are Brauer-Severi spaces $P_\alpha\to
  X$ and $P_\beta\to X$ representing the images of $\alpha$ and
  $\beta$ in $\Br(X)$.  Applying Lemma \ref{L:br-sev-sch-leray} to the
  maps $P_\alpha\times_X P_\beta\to P_\alpha\to X$, we see that it
  suffices to prove the lemma under the additional assumption that the
  Brauer classes associated to $\alpha$ and $\beta$ are trivial.
  Thus, there are invertible sheaves $L_\alpha$ and $L_\beta$ with
  $\alpha=c_1(L_\alpha)$ and $\beta=c_1(L_\beta)$.  Using the Kummer
  sequence, we see that the support hypothesis is equivalent to the
  statement that for an algebraic space $T\to X$, the preimage
  $L_\alpha|_T$ is in $n\Pic(T)$ if and only if $L_\beta|_T$ is in
  $n\Pic(T)$.  Let $T=\mathbf V(L_\alpha)^{\ast}$ as in Lemma
  \ref{L:pic-mod}, so that $\Pic(T)=\Pic(X)/\langle L_\alpha\rangle$.
  We conclude that $L_\beta\in\langle L_\alpha\rangle + n\Pic(X)$.
  Reversing the roles of $\alpha$ and $\beta$, we find that
  $L_\alpha\in\langle L_\beta\rangle + n\Pic(X)$.  It follows that the
  images of $L_\alpha$ and $L_\beta$ generate the same cyclic subgroup
  of $\Pic(X)/n\Pic(X)$, and this yields the result.
\end{proof}

Another way to understand the statement of Lemma \ref{L:cyclic-gp-muffin} is
that there is an automorphism $\psi:\m_n\to\m_n$ such that
$\psi^{\ast}\beta=\alpha$.  In this form, the statement obviously
generalizes to diagonalizable finite group schemes.

\begin{cor}\label{C:diagble-gp-muffin}
  Let $X$ be a strongly R1 algebraic space and $D$ a diagonalizable
  finite group scheme.  If $\alpha,\beta\in\H^2(X,D)$ are Brauer
  cohomology classes with the same
  support then there is an automorphism $\psi:D\simto D$ such that
  $\psi^{\ast}\beta=\alpha$.
\end{cor}
\begin{proof}
  The proof is immediate, since $D$ breaks up as a finite product of group
  schemes of the form $\m_n$ and $\psi$ can be defined on each factor.
\end{proof}

Using Corollary \ref{C:diagble-gp-muffin}, we will 
prove Theorem \ref{thm:summary}(4).  It is 
important to note that by forgetting the automorphism data, there is no 
hope of recovering the gerbe structure, which consists of a specified
trivialization of the inertia stack.  Thus, the best we can hope for
is recovery of the abstract stack, and this is indeed possible in
certain situations.

\begin{lem}
  Suppose $X$ is an algebraic space and $P\to X$ is faithfully flat with
  geometrically connected fibers.  For any finite \'etale group space
  $G\to X$, the
  natural map $\H^1(X,G)\to\H^1(P,G)$ is injective.
\end{lem}
\begin{proof}
  By Lemma \ref{L:section-lemma}, given two $G$-torsors $T$ and $T'$ on $X$,
  the finite \'etale $X$-space $\Isom_G(T,T')$ has a section if and
  only its pullback to $P$ has a section.  The result follows.
\end{proof}

\begin{lem}\label{L:faithfully-flat-splitting}
  Let $X$ be a strongly R1 algebraic space and $D$ a diagonalizable
  finite group scheme.  Given a Brauer class $\alpha\in\H^2(X,A)$, there is
  a faithfully flat morphism $P_\alpha\to X$ with geometrically
  connected fibers such that $\alpha|_{P_\alpha}=0\in\H^2(P_\alpha,A)$.
\end{lem}
\begin{proof}
  Writing $D$ as a direct sum of group schemes of the form $\m_n$, it
  immediately follows that we may assume $D=\m_n$.  The class $\alpha$
  has an image $\widebar\alpha\in\Br'(X)=\Br(X)$ (since $X$ is
  tasty).  Let $P_0\to X$ be a Brauer-Severi space representing
  $\widebar\alpha$, so that $\alpha|_{P_0}$ is the first Chern class of an
  invertible sheaf $L\in\Pic(P_0)$.  Applying Lemma \ref{L:pic-mod}, we see
  that there is a faithfully flat map $P_\alpha\to P_0$ such that $L$
  becomes an $n$th power (in fact, trivial) on $P_\alpha$.  It follows
  that $\alpha|_{P_\alpha}=0$, as required. 
\end{proof}

\begin{prop}\label{P:const-band-recov}
  Let $D$ be a diagonalizable finite group scheme and let $\ms X$ be a
  $D$-gerbe over a strongly R1 algebraic space $X$.  If $\ms Y$ is a
  quasi-algebraic stack and $F_{\ms X}$ is isomorphic to $F_{\ms Y}$
  then $\ms X$ is isomorphic to $\ms Y$.
\end{prop}
\begin{proof}
  Since $F_{\ms X}$ and $F_{\ms Y}$ are isomorphic, we know by
  Corollary \ref{cor:summary-absolute}(3) that there is a $1$-morphism
  $\ms Y\to\ms X$ making $\ms Y$ a gerbe. By Proposition
  \ref{sec:classifying-stacks-8} the automorphism groups at geometric
  points of $\ms Y$ are all (non-canonically) isomorphic to $A$.  It
  follows that the band $G$ of $\ms Y/X$ is a form of $D_X$.  Since
  $D$ is abelian (so that bands and groups are equivalent), we have
  that $G$ is classified by an element of $\H^1(X,\Aut(D))$.  Since
  $\Aut(D)$ is a finite \'etale group scheme, it follows from Lemma
  \ref{L:section-lemma} that we can detect triviality of this
  cohomology class after pulling back along any faithfully flat
  morphism with geometrically connected fibers.  Thus, applying Lemma
  \ref{L:faithfully-flat-splitting}, we may pull back to $P\to X$ so
  that $\ms X|_P\cong\B{D}$.  In this case, $\ms Y$ also has a global
  section, via $\phi$, so that $\ms Y\cong\B{G}$.  Now the triviality
  of the band follows from Proposition \ref{sec:classifying-stacks-8}.
  We may thus choose an identification of the band of $\ms Y$ (on $X$,
  by applying Lemma \ref{L:section-lemma} to $\isom_X(G,D)$) with $D$.

  Write $\alpha\in\H^2(X,D)$ for the class corresponding to $\ms X$
  and $\beta\in\H^2(X,D)$ for the class corresponding to $\ms Y$.  Via
  $\phi$, we see that $\alpha$ and $\beta$ have the same support.
  Using the fact that $\alpha$ is Brauer, this then implies that
  $\beta$ is also Brauer. 
  By Corollary \ref{C:diagble-gp-muffin} there is an isomorphism $\psi:D\simto
  D$ such that $\psi^{\ast}\beta=\alpha$.  Composing the
  $D$-structure on $\ms Y$ with $\psi^{-1}$ produces a new $D$-gerbe
  structure on $\ms Y$ for which the cohomology classes associated to
  $\ms X$ and $\ms Y$ agree.  By Giraud's fundamental theorem 
  (Theorem 3.4.2(i) of Chapter IV of \cite{gi4}),
  there is an ($D$-linear) isomorphism $\ms X\to\ms Y$ over $X$,
  yielding the result.
\end{proof}

\subsection{$\G_m$-gerbes}
\label{sec:g_m-gerbes}

Let $X$ be an algebraic space and $\ms X_i\to X$, $i=1,2$ two
$\G_m$-gerbes such that $[\ms X_1]\in\Br(X)\subseteq\H^2(X,\G_m)$.
Write $n$ for the order of $[\ms X_1]$. Our final task is to prove
Theorem \ref{thm:summary}(5), which is the following statement.

\begin{prop}\label{sec:g_m-gerbes-1}
  If $F_{\ms X_1}$ and $F_{\ms X_2}$ are isomorphic then $\ms X_1$
  and $\ms X_2$ are isomorphic. 
\end{prop}
In other words, $[\ms X_1]=\pm[\ms X_2]$ in $\H^2(X,\G_m)$.

\begin{lem}\label{sec:g_m-gerbes-2}
  If $\ms Y\to X$ and $\ms Z\to X$ are abelian gerbes, then there is
  a natural map
$$\isom_X(F_{\ms Y},F_{\ms Z})\to\isom(L(\ms Y),L(\ms Z)),$$
where $L(\ms Y)$ and $L(\ms Z)$ denote the bands of the gerbes.
\end{lem}
\begin{proof}
  This follows from the proof of Corollary \ref{cor:band-love}.
\end{proof}

Let $\pi:P\to X$ be a Brauer-Severi space with Brauer class $[\ms
X_1]$.  We can view $\pi$ as a \v Cech covering in the fppf topology.
It is elementary that $\vH^2(\{P\to X\},\G_m)=0$.  The \v
Cech-to-derived spectral sequence thus yields an exact sequence
\begin{equation}\label{eq:1}
0\to\vH^1(\{P\to X\},\ms H^1(\G_m))\to\H^2(X,\G_m)\to\vH^0(\{P\to X\},\ms H^2(\G_m)).
\end{equation}

\begin{lem}\label{sec:g_m-gerbes-3}
  The natural map $\vH^0(\{P\to X\},\ms H^2(\G_m))\to\H^2(P,\G_m)$
  is an isomorphism.
\end{lem}
\begin{proof}
  The presheaf $\ms H^2(\G_m)$ assigns to an $X$-space $Y$ the group
  $\H^2(Y,\G_m)$.  We get $\vH^0$ of the presheaf by forming the
  equalizer of the diagram
$$\xymatrix{\H^2(P,\G_m)\ar@<.5ex>[r]^-{\pr_1^\ast}\ar@<-.5ex>[r]_-{\pr_2^{\ast}}&\H^2(P\times_X
  P,\G_m).}$$ The Leray spectral sequence shows that the pullback map
$\pi^{\ast}:\H^2(X,\G_m)\to\H^2(P,\G_m)$ is surjective.  Since
$\pi\pi_1=\pi\pr_2$, we see that $\pr_1^{\ast}=\pr_2^{\ast}$, so that
the equalizer is $\H^2(P,\G_m)$, as desired.
\end{proof}

\begin{cor}\label{sec:g_m-gerbes-4}
  There is a natural isomorphism $\vH^1(\{P\to X\},\ms H^1(\G_m))\simto\Z/n\Z$
  onto the subgroup of $\H^2(X,\G_m)$ generated by $[\ms X_1]$.
\end{cor}
\begin{proof}
  By Lemma \ref{sec:g_m-gerbes-3}, $\vH^1(\{P\to X\},\ms H^1(\G_m))$ is
  identified (via the edge map in the spectral sequence) with the
  kernel of the pullback map $\H^2(X,\G_m)\to\H^2(P,\G_m)$.  The
  argument cited in the proof of Lemma \ref{L:br-sev-sch-leray}
  shows that this kernel is precisely the subgroup generated by
  $[P]=[\ms X_1]$.
\end{proof}

Let $\widebar x\to X$ be a geometric point.  There is a canonical
isomorphism 
\begin{equation}
  \label{eq:3}
  \Z\simto\Pic(P_{\widebar x})
\end{equation}
given by the unique ample generator.

\begin{lem}\label{sec:g_m-gerbes-5}
  There is a natural injection
$$\vH^1(\{P\to X\},\ms H^1(\G_m))\inj(\Z\times\Z)/(n,-n)$$
onto the subgroup spanned by $(1,-1)$ arising from the restriction of
cocycles to the fiber of $P\times_X P$ over $\widebar x$ and the
isomorphism of equation \eqref{eq:3} above.
\end{lem}
\begin{proof}
  The \v Cech cohomology group is the cohomology group at the middle
  node of the diagram
$$\xymatrix{\Pic(P)\ar@<.5ex>[r]\ar@<-.5ex>[r]&\Pic(P\times_XP)\ar@<1ex>[r]\ar[r]\ar@<-1ex>[r]&\Pic(P\times_XP\times_XP)}.$$
We know that the relative Picard space of an $\ell$-fold product of
$P$ over $X$ is $\Z^\ell$.  Sending a cohomology class to the
associated section of the relative Picard space yields a diagram
\begin{equation}
  \label{eq:4}
  \xymatrix{0\ar[d]&0\ar[d]&0\ar[d]\\
\Pic(X)\ar[d]\ar@<.5ex>[r]\ar@<-.5ex>[r]&\Pic(X)\ar[d]\ar@<1ex>[r]\ar[r]\ar@<-1ex>[r]&\Pic(X)\ar[d]\\
\Pic(P)\ar[d]\ar@<.5ex>[r]\ar@<-.5ex>[r]&\Pic(P\times_XP)\ar[d]\ar@<1ex>[r]\ar[r]\ar@<-1ex>[r]&\Pic(P\times_XP\times_XP)\ar[d]\\
\Z\ar@<.5ex>[r]\ar@<-.5ex>[r]&\Z^2\ar@<1ex>[r]\ar[r]\ar@<-1ex>[r]&\Z^3}
\end{equation}
with exact columns and whose top and bottom rows are acyclic at the
middle node.  Moreover, the bottom vertical maps agree with the ones
given by restricting to the geometric fibers over $\widebar x$.  A
straightforward calculation shows that the horizontal 
kernel at $\Z^2$ is the subgroup generated by $(1,-1)$.  In addition,
it is standard that the image of $\Pic(P)$ in $\Z$ is the subgroup
generated by $n$.  This yields a map 
$$\vH^1(\{P\to X\},\ms H^1(\G_m))\to \Z^2/(n,-n),$$
which we claim is an isomorphism onto the subgroup generated by
$(1,-1)$.  That the image lies in that subgroup follows from the
preceding sentences.  By Corollary \ref{sec:g_m-gerbes-4}, it is
enough to show that $(1,-1)$ is in the image of $\Pic(P\times_X
P)\to\Z^2$; a simple chase through diagram \eqref{eq:4} then shows
that one can find such an element which is in the horizontal kernel at
$\Pic(P\times_X P)$.

We claim that in fact the map $\Pic(P\times_X P)\to\Z^2$ is
surjective.  To see this, first note that $P\times_X P\to P$ (by
either projection) is a trivial Brauer-Severi space, so that there is
an invertible sheaf $\ms L$ on $P\times_X P$ which is a relative $\ms
O(1)$ for such a projection.  More precisely, the canonical relatively
ample section $1$ in $\Z_X$ (the relative Picard space) lifts to a
section of the Picard stack (i.e., an invertible sheaf) upon pullback
to $P$, and it is this sheaf which we take for $\ms L$.  It follows
from this description that the restriction of $\ms L$ to the geometric
fiber over $\widebar x$ has class $(1,0)$.  Similarly, we can find
$\ms M$ mapping to $(0,1)$.  The claim follows, and with it the
result.
\end{proof}

\begin{lem}\label{sec:g_m-gerbes-6}
  Given an isomorphism $\gamma:F_{\B\G_m}|_{P\times_X P}\simto
  F_{\B\G_m}|_{P\times_X P}$ such that $\gamma(P\times_X
  P)(\ast)=\ast$, we have that $\gamma(P_{\widebar x}\times_{\widebar
    x} P_{\widebar x})=\id$ or $-\id$ as automorphisms of the set
  $\Z^{\oplus 2}=F_{\B\G_m}(P_{\widebar x}\times_{\widebar x}
  P_{\widebar x})$.
\end{lem}
\begin{proof}
  Choosing a section $y\in P_{\widebar x}(\widebar x)$ yields two maps
  $P_{\widebar x}\to P_{\widebar x}\times_{\widebar x} P_{\widebar x}$
  (identifying the fibers of the projections over $y$) which, by cohomology
  and base change, yield an isomorphism $\Pic(P_{\widebar
    x}\times_{\widebar x} P_{\widebar x})\simto\Pic(P_{\widebar
    x})\times\Pic(P_{\widebar x})$.  Composing this with pullback
  along the diagonal embedding $P_{\widebar x}\to P_{\widebar x}\times_{\widebar x}
  P_{\widebar x}$ yields the group law on
  $\Pic(P_{\widebar x})=\Z$.  Since
  all of the maps in question are derived from geometric
  constructions, they are $\gamma$-equivariant (as maps of sets).  It
  follows using the fact that $\gamma(\ast)=\ast$ that the action of
  $\gamma$ on $F_{\B\G_m}(P_{\widebar x})$ is by a group automorphism
  of $\Z$.  Thus, $\gamma$ acts on $F_{\B\G_m}(P_{\widebar
    x})$ as $\id$ or $-\id$.  Since the proof produces a $\gamma$-equivariant
  isomorphism $F_{\B\G_m}(P_{\widebar x}\times_{\widebar x} P_{\widebar x})\simto
  F_{\B\G_m}(P_{\widebar x})\times F_{\B\G_m}(P_{\widebar x})$, the
  result follows.
\end{proof}

\begin{cor}\label{sec:g_m-gerbes-7}
  Let $a$ and $b$ be two elements of $F_{\ms X_2}(P\times_X P)$. If
  $$f,g:F_{\B\G_m}|_{P\times_X P}\simto F_{\ms X_2}|_{P\times_X P}$$ are
  two isomorphisms sending $\ast$ to $a$ then $f^{-1}(b)|_{(P\times
    P)_{\widebar x}}=\pm g^{-1}(b)|_{(P\times P)_{\widebar x}}$.
\end{cor}
\begin{proof}
  Consider the diagram of isomorphisms
$$\xymatrix{F_{\B\G_m}\ar[dr]^f\ar[dd]_\gamma&\\
&F_{\ms X_2}\\
F_{\B\G_m},\ar[ur]_g&}$$
where $\gamma=g^{-1}f$.  
By assumption $\gamma(P\times_X P)(\ast)=\ast$.  We conclude by Lemma
\ref{sec:g_m-gerbes-6} that $\gamma(P_{\widebar x}\times_{\widebar x} P_{\widebar
  x})=\pm\id$, so that $g(P_{\widebar x}\times_{\widebar x} P_{\widebar
  x})=\pm f(P_{\widebar x}\times_{\widebar x} P_{\widebar x})$.  Thus, we conclude
that $f^{-1}(b)|_{(P\times P)_{\widebar x}}=\pm g^{-1}(b)|_{(P\times
  P)_{\widebar x}}$.  
\end{proof}

\begin{lem}\label{sec:g_m-gerbes-8}
  An element $F_{\ms X_i}(T)$ determines a unique isomorphism
  $F_{\B\G_{m}|_T}\simto F_{\ms X_i|_T}$ whose induced morphism of
  bands is $\id:\G_m\to\G_m$.  
\end{lem}
\begin{proof}
  Given an invertible sheaf $\ms L$ and an object $\sigma$ of $(\ms
  X_i)_T$, there results a new object $\sigma\tensor\ms L$ of $(\ms
  X_i)_T$ by standard methods (e.g., one can use the descent datum for
  $\ms L$ and the fact that $\ms X_i$ is a $\G_m$-gerbe to produce a
  form of $\sigma$ with the same cocycle). Moreover, replacing $\sigma$ by an
  isomorphic object $\sigma'$ yields an isomorphic object
  $\sigma'\tensor\ms L$, and likewise for $\ms L$.  Finally, any
  object $\tau$ of $(\ms X_i)_T$ has the form $\sigma\tensor\ms L$ for
  a unique $\ms L$.  These statements are also functorial in $T$.  The
  result follows.
\end{proof}

\begin{lem}\label{sec:g_m-gerbes-9}
  Let $\alpha$ be an element of $F_{\ms X_2}(P)$.  Suppose the two
  pullbacks of $\alpha$ to $F_{\ms X_2}(P_{\widebar x}\times_{\widebar x}
  P_{\widebar x})$ are $a$ and $b$.  For any isomorphism
  $\phi:F_{\B\G_m\times(P_{\widebar x}\times_{\widebar x} P_{\widebar x})}\simto
  F_{\ms X_2\times (P_{\widebar x}\times_{\widebar x} P_{\widebar x})}$ sending
  $\ast$ to $a$, we have that the image of $\phi^{-1}(b)$ under the map
$$\Pic(P_{\widebar x}\times_{\widebar x} P_{\widebar x})\to (\Z\times \Z)/(n,-n)$$ 
represents either $[\ms X_2]$ or $-[\ms X_2]$ via the injection of
Lemma \ref{sec:g_m-gerbes-5}.
\end{lem}
\begin{proof}
In this proof we freely use the theory of twisted sheaves as developed
in \cite{li1}.  
Suppose first that $\phi$ is the restriction of the isomorphism
arising from $\alpha$ as in Lemma \ref{sec:g_m-gerbes-8}.  We can
compute $\phi^{-1}(b)$ as follows: the element $\alpha$ corresponds to
an $\ms X_2$-twisted invertible sheaf $\ms L$ on $P$ (up to isomorphism), and
$\phi^{-1}(b)$ is the isomorphism class of the restriction of
$\pr_1^{\ast}\ms L\tensor\pr_2^{\ast}\ms L^{\vee}$ to $P_{\widebar
  x}\times_{\widebar x}P_{\widebar x}$.  As an element of $\Z^{\oplus
  2}$, this lies in the kernel of the coboundary map in the bottom row of
diagram \eqref{eq:4} above and is thus a multiple of $(1,-1)$ (i.e., a
cocycle).
Furthermore, changing $\ms L$ by an invertible sheaf on $P$ changes
the resulting cocycle by a coboundary, leaving the cohomology class
(in $\vH^1(\{P\to X\},\ms H^1(\G_m))$ invariant).

On the other hand, since $\ms X_2$ is trivial on $P$, we know by Lemma
\ref{L:br-sev-sch-leray} that $[\ms X_2]=d[\ms X_1]$ for some $d$, and
this implies that we can identify $\ms X_2$-twisted sheaves with
$d$-fold $\ms X_1$-twisted sheaves.  In particular, if $\ms M$ is an
$\ms X_1$-twisted invertible sheaf on $P$, we may assume (for the
purposes of computing the cohomology class) that $\ms L=\ms M^{\tensor
  d}$.  But then we find that $\phi^{-1}(b)$ is $d$ times the class of
$\pr_1^{\ast}\ms M\tensor\pr_2^{\ast}\ms M^{\vee}$.  Since the latter
is precisely the image of $[\ms X_1]$ in $\vH^1(\{P\to X\},\ms
H^1(\G_m))$, the result then follows from Lemma \ref{sec:g_m-gerbes-5}.

When $\phi$ is not the isomorphism induced by $\alpha$, we can apply
Corollary \ref{sec:g_m-gerbes-7} to compare the two, yielding the
possible change of sign.
\end{proof}

\begin{proof}[Proof of Proposition \ref{sec:g_m-gerbes-1}] 
  Let $\psi:F_{\ms X_1}\to F_{\ms X_2}$ be an isomorphism and let
  $\pi:P\to X$ be a Brauer-Severi space with cohomology class $[\ms
  X_1]$, as above.  The isomorphism $\psi$ induces an isomorphism of
  bands $L_{\psi}:L_1\simto L_2$, so that $\ms X_2$ is also a
  $\G_m$-gerbe.  

  Since $\ms X_1$ is trivial on $P$, there is an isomorphism
  $F_{\B\G_m\times P}\simto F_{\ms X_1\times_X P}$ coming from an
  element $\alpha\in F_{\ms X_1}(P)$.  Composing with $\psi$ yields an
  isomorphism $F_{\B\G_m\times P}\to F_{\ms X_2\times P}$ sending
  $\ast$ to $\psi(\alpha)$.  By Lemma \ref{sec:g_m-gerbes-9} (applied
  to the pairs $\ms X_1,\ms X_1$ and $\ms X_2,\ms X_2$), we know that
  the two preimages of $\alpha$ (resp.\ $\psi(\alpha)$) in
  $F_{\B\G_m}(P_{\widebar x}\times_{\widebar x} P_{\widebar x})$
  differ by a representative for the cohomology class of $[\ms X_1]$
  (resp.\ $\pm[\ms X_2]$) in $\vH^1(P,\ms H^1(\G_m))$.  By Corollary
  \ref{sec:g_m-gerbes-4}, we conclude that $[\ms X_1]=\pm[\ms X_2]$,
  as desired (as the change of sign corresponds to changing the
  trivialization of the band of $\ms X_2$ and does not change the
  underlying stack structure).
\end{proof}

\subsection{Counterexamples}\label{sec:context}

On sites smaller than $S\Sch$, one can construct various examples 
of stacks which are not isonatural.  

(1) On the small \'etale site of an algebraically closed field $k$,
the stacks $\B{G}$ for any group $G$ all have the singleton sheaf as
associated functor.  When $G$ is finite, these stacks have
representable diagonals, satisfy the kind of limiting property we
require for quasi-algebraic stacks, etc.  In this case, the underlying
site clearly does not contain the kind of ``anabelian'' structures
needed to reconstruct anything.

(2) If $X$ is a geometrically unibranch scheme and $F$ is a (discrete) torsion
free abelian group, then $\B{F}$ again has associated functor
represented by $X$.  The diagonal is again
representable, and the diagonal of $\B{F}$ again satisfies the desired
limiting property with respect to inverse systems of objects of the
small \'etale site of $X$.

(3) The small Zariski site also lacks anabelian structure: the 
stack $\B{\G_m}$ has singleton associated functor on the small
Zariski site of $\A^1$, while $\B{G}$ has singleton associated functor
on the small Zariski site of any irreducible scheme for any discrete group $G$.

(4) Using the techniques developed in Section \ref{sec:g_m-gerbes},
one can also make families of examples where the associated functor is
(marginally) more complicated.  Let us show that as long as there is
an element $\alpha\in\Br(k)$ of order invertible in $k$ and larger
than $4$, there are $\G_m$-gerbes $\ms X$ and $\ms Y$ on the small
\'etale site of $k$ such that $F_{\ms X}$ is isomorphic to $F_{\ms Y}$
but $\ms X$ is not isomorphic to $\ms Y$ (as a stack).

There are two important sheaves on $k_{\et}$: the sheaf of
multiplicative groups $\G_m$ and the sheaf $\m_{\infty}$ of all roots
of unity.  It is clear that restriction defines a natural map
$\saut(\G_m)\to\saut(\m_{\infty})$.  Kummer theory shows that the
natural inclusion $\mu_{\infty}\inj\G_m$ induces an isomorphism
$\H^2(k_{\et},\m_{\infty})'\simto\H^2(k_{\et},\G_m)'$ (where the
superscripts indicate the prime-to-characteristic parts of the groups
in question).

\begin{lem}\label{sec:count-diff-sites-1}
  The natural map $\Z/2\Z\to\saut(\m_{\infty})$ sending $1$ to 
  inversion is an isomorphism.
\end{lem}
\begin{proof}
  With respect to any chosen (non-canonical) isomorphism
  $\m_{\infty}(\widebar k)\simto\widehat{\Z}$, the sheaf of
  automorphisms gets identified with continuous Galois-equivariant
  automorphisms of $\widehat{\Z}$.  But the continuous automorphism
  group of $\widehat{\Z}$ is $\Z/2\Z$, generated by inversion.  Since
  inversion is clearly Galois-equivariant, the result follows.
\end{proof}

\begin{cor}\label{sec:count-diff-sites-2}
  The orbits of the action of $\Aut(\G_m)$ on $\Br(k)'$ are given by
  $\{\alpha,-\alpha\}$ for $\alpha\in\Br(k)'$.
\end{cor}
\begin{proof}
  Because every automorphism of $\G_m$ induces an automorphism of
  $\m_{\infty}$ compatibly with the respective actions on cohomology, the
  corollary follows from Lemma \ref{sec:count-diff-sites-1} and the fact 
  that $\H^2(k_{\et},\m_{\infty})'\to\H^2(k_{\et},\G_m)'$ is an
  isomorphism.
\end{proof}

Given an element $\alpha\in\Br(k)$, write $\langle\alpha\rangle$ for
the subgroup generated by $\alpha$.

\begin{prop}
  Suppose $\alpha\in\Br(k)'$ has order larger than $4$, so that the
  generators for $\langle\alpha\rangle$ lie in at least two orbits
  under the automorphism group of $\G_m$.  Then there are two
  non-isomorphic stacks $\ms X$ and $\ms Y$ such that $F_{\ms X}$ and
  $F_{\ms Y}$ are isomorphic.
\end{prop}
\begin{proof}
  Let $\beta$ be a generator for $\langle\alpha\rangle$ which is
  distinct from $\alpha$ and $-\alpha$.  Let $\ms X$ be a $\G_m$-gerbe
  representing $\alpha$ and $\ms Y$ a $\G_m$-gerbe representing
  $\beta$.  It is elementary that the support of $\alpha$ and $\beta$
  in $k_{\et}$ are the same.  On the other hand, if $F_{\ms
    X}(L)\neq\emptyset$ then it is a singleton (since two sections of
  $\B\G_m$ differ by an invertible sheaf, of which there is only one
  on $L_{\et}$), which means that $F_{\ms X}$ and $F_{\ms Y}$ are
  (even canonically!) isomorphic.

  If $\ms X\simto\ms Y$ is an isomorphism, then it induces an
  isomorphism of the bands.  Changing the trivialization of the band
  of $\ms Y$ by an automorphism of $\G_m$ yields two $\G_m$-gerbes
  with the same cohomology class.  But we know from Corollary
  \ref{sec:count-diff-sites-2} that the orbit of $\beta$ for the
  automorphism group of $\G_m$ is $\{\beta,-\beta\}$.  Since $\alpha$
  is neither $\beta$ nor $-\beta$, we see that there cannot be an
  isomorphism between $\ms X$ and $\ms Y$.
\end{proof}


\end{document}